\documentclass[12pt,french,american]{article}
\usepackage[margin=1in]{geometry}
\usepackage{mathrsfs, bbm, dsfont}
\usepackage{
amsmath,
amsfonts,
latexsym, 
amsrefs,
amssymb,
mathtools,
tikz,
tabularx,
minibox,
amsthm,
}
\usepackage{bm}
\usepackage{babel}
\usepackage{floatrow}
\usepackage{comment}
\usepackage[T1]{fontenc}
\usepackage{graphicx}
\DeclareGraphicsExtensions{.eps}
\usepackage{url}
\usepackage{wrapfig}
\usepackage{titlesec}
\usepackage{titletoc}
\usepackage[title,titletoc]{appendix}
\usepackage{pdflscape}
\usepackage{afterpage}
\usepackage{tocloft}

\usepackage{caption}
\usepackage{subcaption}
\usepackage{enumerate}
\definecolor{mygray}{gray}{0.3} 
\usepackage[colorlinks=true,urlcolor=gray,linkcolor=mygray,citecolor=mygray]{hyperref}

\usetikzlibrary{shapes}
\usetikzlibrary{intersections}
\usetikzlibrary{svg.path}
\usetikzlibrary{decorations.markings}
\usetikzlibrary{hobby}

\numberwithin{equation}{section}
\renewcommand{\rm}{\mathrm}

\newcommand{\al}{\alpha}
\newcommand{\de}{\delta}
\newcommand{\be}{\begin{equation}}
\newcommand{\ee}{\end{equation}}

\newcommand{\Ga}{{\Gamma}}
\newcommand{\la}{\lambda}

\newcommand{\cop}{\mathrm{cop}}

\renewcommand{\S}{\mathbb S}

\newcommand{\D}{\mathbb{D}}

\renewcommand{\S}{[\bf S]}

\newcommand{\Pair}{\mathscr{P}}
\newcommand{\AdPair}{\mathscr{P}}
\newcommand{\Gin}{\rm{Gin}}

\newcommand{\CUE}{\mathrm{CUE}}

\newcommand{\distconv}{\xrightarrow[N \rightarrow \infty]{d}}

\renewcommand{\S}{\mathbb{S}^2}
\newcommand{\Proj}{\mathbb{R}P^2}
\newcommand{\OO}{\mathrm{O}}
\newcommand{\oo}{\mathrm{o}}

\DeclareMathOperator{\PP}{\mathbb{P}}

\newcommand{\E}{\mathbb{E}}
\newcommand{\R}{\mathbb{R}}
\newcommand{\C}{\mathbb{C}}
\newcommand{\N}{\mathbb{N}}
\newcommand{\Z}{\mathbb{Z}}

\newcommand{\dd}{\mathrm{d}}
\newcommand{\tr}{\mathrm{Tr}}
\newcommand{\Tr}{\mathrm{Tr}}
\renewcommand{\lll}{\left}
\newcommand{\rr}{\right}
\newcommand{\FC}{\mathrm{FC}}

\theoremstyle{plain} 
\newtheorem{theorem}{Theorem}[section]
\newtheorem*{theorem*}{Theorem}
\newtheorem{lemma}[theorem]{Lemma}

\newtheorem*{lemma*}{Lemma}
\newtheorem{corollary}[theorem]{Corollary}
\newtheorem*{corollary*}{Corollary}
\newtheorem{proposition}[theorem]{Proposition}
\newtheorem*{proposition*}{Proposition}
\newtheorem{fact}[theorem]{Fact}
\newtheorem*{fact*}{Fact}

\newtheorem{definition}[theorem]{Definition}
\newtheorem*{definition*}{Definition}

\newtheorem*{example*}{Example}
\newtheorem{remark}[theorem]{Remark}

\newtheorem*{remark*}{Remark}
\newtheorem*{remarks*}{Remarks}

\makeatletter
\renewcommand{\section}{\@startsection
{section}
{1}
{0mm}
{-2\baselineskip}
{1\baselineskip}
{\normalfont\large\scshape\centering}} 
\makeatother

\makeatletter
\renewcommand{\subsection}{\@startsection
{subsection}
{2}
{0mm}
{-\baselineskip}
{1 \baselineskip}
{\normalfont\scshape}} 
\makeatother

\makeatletter
\renewcommand{\subsubsection}{\@startsection{subsubsection}{3}{\z@}%
  {3.25ex \@plus 1ex \@minus .2ex}{-1em}{\normalfont\normalsize\itshape}}
\makeatother
\usepackage[T1]{fontenc}

\def\@empty{}

\def\author#1{\par
    {\centering{\authorfont#1}\par\vspace*{0.05in}}}

\def\titlefont{\fontsize{12}{15} \textbf{}}
\def\authorfont{\fontsize{12}{15}}

\let\affiliationfont\rhfont

\def\address#1{\par
    {\centering{\affiliationfont#1\par}}\par\vspace*{12pt}
}

\def\body{
\setcounter{footnote}{0}
\def\thefootnote{\alph{footnote}}
\def\@makefnmark{{$^{\rm \@thefnmark}$}}
}

\def\title#1{ 
    \thispagestyle{plain}
    \vspace{-30pt}
    \vskip 79pt
    {\centering{\titlefont #1\par}}%
    \vskip 1em
}

\providecommand{\keywords}[1]{\vspace{.3in} \textbf{Keywords:} #1}

\begin{document}

\title{\uppercase{ \textbf{On Words of non-Hermitian Random Matrices } } }
\vspace{.5cm}
\noindent \begin{minipage}[b]{0.49\textwidth}
 \author{Guillaume Dubach \\
 IST Austria}
\address{\small guillaume.dubach@ist.ac.at}
\end{minipage}
\begin{minipage}[b]{0.49\textwidth}
\author{Yuval Peled  \\
 Courant Institute, NYU}
\address{\small yuval.peled@cims.nyu.edu}
\end{minipage}

\vspace{-.1in}
\begin{abstract}
We consider words $G_{i_1} \cdots G_{i_m}$ involving i.i.d. complex Ginibre matrices, and study tracial expressions of their eigenvalues and singular values. We show that the limit distribution of the squared singular values of every word of length $m$ is a Fuss-Catalan distribution with parameter $m+1$. This generalizes previous results concerning powers of a complex Ginibre matrix, and products of independent Ginibre matrices. In addition, we find other combinatorial parameters of the word that determine the second-order limits of the spectral statistics. For instance, the so-called coperiod of a word characterizes the fluc\-tuations of the eigenvalues. We extend these results to words of general non-Hermitian matrices with i.i.d. entries under moment-matching assumptions, band matrices, and sparse matrices.

These results rely on the moments method and genus expansion, relating Gaussian matrix integrals to the counting of compact orientable surfaces of a given genus. This allows us to derive a central limit theorem for the trace of any word of complex Ginibre matrices and their conjugate transposes, where all parameters are defined topologically.


\end{abstract}

\vspace{-.2in}
\tableofcontents


\section{Introduction}

The connection between matrix integrals and enumeration of combinatorial maps on surfaces was initially discovered in the context of quantum field theory \cites{feynman, hooft, parisi}. In mathematics, the problem of enumerating  combinatorial maps and  planar triangulations was introduced by Tutte \cites{tutte,tutte1}, and its connection to random matrix theory appeared in the work of Harer and Zagier concerning the Euler characteristic of moduli spaces of complex curves \cite{harer}. See \cites{Guionnet2006,zvonkin,book} for a detailed exposition. \medskip

The basic idea that connects matrix integrals and enumeration of maps is Wick's principle. Roughly speaking, the principle states that a moment of many Gaussian variables can be expanded to a sum of moments of pairs of variables, where the summation is taken over all the pairings. Wick's principle can be used in the computation of moments of traces of Gaussian matrices. In such case, each pairing in Wick's summation naturally gives rise to a surface, and it turns out that the contribution of any pairing to the sum is determined by the number of connected components of this surface, and their genera. \medskip

In general, there are two directions in which to exploit this connection. On the one hand, it enables to enumerate combinatorial objects like planar maps by solving the corresponding matrix integral. On the other hand, it relates questions in random matrix theory to combinatorial questions about surfaces. For example, ideas related to genus expansion have been used extensively in the context of several-matrix models \cites{CapitaineCasalis1, CapitaineCasalis2, GuionnetMaurel,GuionnetMaida}, and in free probability in order to establish first and second degree freeness of various ensembles \cites{MingoSpeicher1, MingoSpeicher2, NicaSpeicher, Redelmeier}.
For another example, an analogous version of the genus expansion for Haar measured unitary matrices has been recently used by Puder and Magee to study word measures \cite{Puder}. They studied the parameters of a word, considered as an element of a free group, that determine key features of the distribution of the random unitary matrix induced by that word. \medskip

The goal of this paper is to use the genus expansion technique to derive properties of singular values, eigenvalues and mixed moments of words of random non-Hermitian matrices. The main focus will be on words in complex Ginibre matrices and their conjugate transposes: let $N>0$ be an integer and  $G_1,G_2,...$ a sequence of i.i.d.\! Ginibre matrices of order $N$. A $*$-word $w$ is a word over the formal alphabet $\{G_1,G_2,...\}\cup\{G_1^*,G_2^*,...\}$ and it induces a random matrix $G_w$ of order $N$ by replacing each letter of $w$ with the appropriate matrix and taking the matrix product. A word whose letters are taken only from the formal alphabet $\{G_1,G_2,...\}$ is called $*$-free. The main question this paper addresses is: what combinatorial parameters of $w$ determine the distribution of the singular values, eigenvalues and mixed moments of $G_w$?  \medskip

Our first result establishes a first order limit for the singular values of a $*$-free word $G_w$. Namely, with probability $1$, the empirical measure of the squared singular values converges weakly to a limit that only depends on the length of $w$.

\begin{theorem}\label{thm:mainFC}
Let $w$ be a $*$-free word of length $m$, $G_w$ the corresponding product of complex $N \times N$ Ginibre matrices, and $\mu_{ww^*}$ be the empirical measure of the squared singular values of $G_w$. The following weak convergence occurs almost surely:
$$
\mu_{ww^*} \distconv \rho_{FC}^{m+1},
$$
where $\rho_{FC}^{m+1}$ denotes the Fuss-Catalan distribution with parameter $s=m+1$.
\end{theorem}

Fuss-Catalan distributions were known to be the first order limit of squared singular values for products of independent GUE matrices \cite{Alexeev_beamer}, products of independent Ginibre matrices \cite{Penson}, as well as powers of Ginibre matrices \cite{Alexeev}. We establish that it actually holds for every $*$-free word of Ginibre matrices, and universally so (that is, beyond the Gaussian case, see Section \ref{extension_section} ). By a similar analysis, we show that the limit of all the {\em mixed matrix moments} of $G_w$ depends only on the length of $w$. Such mixed moments are deeply connected with both eigenvalues and eigenvectors; they appear naturally in the moments of Girko's hermitized form, or in the expansion of the quaternionic resolvent (see \cite{BurdaSwiech}), and their connection with eigenvector overlaps has been studied by Walters \& Starr in \cite{StarrWalters}. The fact that the first order asymptotics of these mixed moments only depends on the length suggests that the limit of empirical measure of eigenvalues of $G_w$ only depends on the length, as is the case for singular values. 
In fact, it is known that the eigenvalues of the products of $m$ i.i.d.\! Ginibre matrices, as well as of the $m$-th power of a Ginibre matrix, exhibit convergence to the same limit, namely the $m$-twisted circular law, which is the pushforward of the circular law by the $m$-th power map~\cite{BurdaNowak}. \medskip

Our next result deals with the second-order behaviour of the eigenvalues of $G_w$. The {\em coperiod} of a $*$-free word $w$, denoted by $\cop(w)$, is the largest integer $k$ such that $w=u^k$ for some $*$-free word $u$. While the empirical measure of the singular values, and likely of the eigenvalues, of $G_w$ is determined by the length of $w$, we show that the fluctuations of polynomial statistics of its eigenvalues are determined by the coperiod of $w$.

\begin{theorem}\label{StarfreeCLT_joint}
For every $*$-free word $w$ and every integer $k>0$,
$$
\left(\Tr(G_w),\Tr(G_w^2),...,\Tr(G_w^k)\right) \distconv 
\left(Z_1,Z_2,...,Z_k \right),
$$
where $Z_j$ is $\mathscr{N}_{\C}(0, j\cdot\cop(w) )$-distributed and the $Z_j$'s are independent.
\end{theorem}

This is reminiscent of a celebrated result of Diaconis-Shahshahani \cite{DiaconisShahshahani} for a matrix of the Circular Unitary Ensemble (CUE), i.e. a unitary matrix chosen according to the Haar measure on $U_N(\C)$.

\begin{theorem}[Diaconis-Shahshahani] If $U$ is a $\CUE$ matrix,
$$
(\Tr U, \Tr U^2, \dots, \Tr U^m) \distconv (Z_1, Z_2, \dots, Z_m)
$$
where the $Z_j$'s are independent complex Gaussian vectors, $Z_j$ having variance~$j$.
\end{theorem}

A generalization of the Diaconis-Shahshahani limit theorem to words of CUE matrices can be found in \cite{MingoSpeicherCUE} (Corollary 3.14) and \cite{Radulescu}. \medskip

The proofs of these theorems rely on a {\em genus expansion} formula for a $*$-word $w$, from which we derive a central limit theorem for the trace of $G_w$.

\begin{theorem}
\label{thm:clt*}
For every $*$-word $w$, there are integers $a_w,b_w,c_w$ such that
\[
\Tr(G_w) - a_w\cdot N \distconv X+iY,
\]
where $X,Y$ are two independent real centered Gaussian variables with variances $\frac{b_w+c_w}{2}$ and $\frac{b_w-c_w}{2}$  respectively.
\end{theorem}
The integers $a_w,b_w,c_w$ count the number of {\em spheres} that can be obtained by an admissible pairing of one or two polygonal faces whose edges are labelled by the letters of $w$ and $w^*$, as explained in Section \ref{CLTsection}. \medskip
 
A related line of research was explored in the context of free probability. The concept of second-order freeness was first introduced by Mingo and Speicher \cites{MingoSpeicher1,MingoSpeicher2}, and then applied to several ensembles of random matrices. For instance, a similar central limit theorem for words of GUE matrices can be found in \cite{MingoSpeicher2} (Theorem 3.1). In addition, Redelmeier introduced its real and quaternionic analogs, and appropriate notions of second-order freeness have been established for Wishart matrices \cite{RedelmeierWishart}, real Ginibre matrices \cite{Redelmeier}, and quaternionic Gaussian matrices \cite{RedelmeierQGE}.
\medskip

We also extend these results to words of general non-Hermitian matrices with i.i.d.\! entries, band matrices and sparse matrices. Hermitian band matrices have been studied notably by Au \cites{Au2018, Au2019} and Male \cite{Male2017} using the tools of traffic probability. This approach establishes that band and sparse Wigner matrices, under some conditions, have the same statistics as Gaussian matrices. In the non-Hermitian case, we show that if the second moment and the fourth moment of the distribution of the entries matches that of the Gaussian distribution then an approximate genus expansion formula holds (Theorem \ref{thm:fourth_moment_extension}). This implies that all our results concerning the empirical measure of the squared singular values, the mixed-moments and the fluctuations of the eigenvalues apply. In addition, we consider $*$-words of complex Gaussian band matrices, where we prove a CLT for their trace (Theorem \ref{band_CLT}), generalizing one of Jana's recent results \cite{Jana_band}. Finally, we show that a weaker version of our results apply even if we consider $*$-words of sparse random matrices with optimal sparsity parameters (Proposition \ref{pro:MG}). \medskip

The remainder of the paper is organized as follows. In Section \ref{SurfacesSection} we present the genus expansion technique for $*$-words, prove Theorem \ref{thm:clt*} and study the combinatorial problem underlying the computation of the parameters appearing in this theorem. Afterwards, in Section \ref{Applications}, we prove our main results concerning singular values, mixed moments and fluctuation of eigenvalues of words of Ginibre matrices. Finally, in Section \ref{extension_section} we extend these results to words of general non-Hermitian matrices with i.i.d.\! entries, band matrices, and sparse matrices. This organization is summarized in Figure \ref{fig:plan_of_paper}.

\begin{figure}[t!]
\begin{center}
\begin{tikzpicture}[scale=.9]
\begin{scope}[thick,decoration={
    markings,
    mark=at position 0.57 with {\arrow{>}}}
    ] 
\node[draw,fill=white,align=center,rounded corners=5] (A1) at (0,0.5) {Genus expansion formula\\Theorem \ref{genusexpansion}};
\node[draw,fill=white,align=center,rounded corners=5] (B1) at (2,-1.2) {Centered genus expansion\\Lemmas 
\ref{lemma:centralized},\ref{lemma:diatomic}};
\node[draw,fill=white,align=center,rounded corners=5] (C1) at (-4,-3) {Genus expansion for a single $*$-word,\\ Equation (\ref{first_order}) in
Theorem \ref{abcthm}};
\draw[postaction={decorate}] (1.8,-2.3)--(-1,-5);
\node[draw,fill=white,align=center,rounded corners=5] (C2) at (4,-3) {CLT for $\Tr(G_w)$\\  Equation (\ref{second_order}) in Theorem \ref{abcthm}};
\node[draw,fill=white,align=center,rounded corners=5] (D1) at (-6.1,-5) {Mixed moments\\Theorem \ref{thm:mixed}};
\node[draw,fill=white,align=center,rounded corners=5] (D2) at (-2.1,-5) {Fuss-Catalan limit\\ Theorem \ref{thm:mainFC}};
\node[draw,fill=white,align=center,rounded corners=5] (D3) at (2,-5) {Coperiod\\
Proposition \ref{StarfreeCLT}};
\node[draw,fill=white,align=center,rounded corners=5] (D4) at (6.7,-5) {Asymptotic independence\\
Proposition \ref{independence}};
\draw[postaction={decorate}] (A1) -- (C1);
\draw[postaction={decorate}] (A1)--(B1);
\draw[postaction={decorate}] (B1)--(C2);
\draw[postaction={decorate}] (C1)--(D1);
\draw[postaction={decorate},color=red] (C1)--(D2);
\draw[postaction={decorate}] (C2)--(D3);
\draw[postaction={decorate}] (C2)--(D4);
\draw[-,color=red,rounded corners=15] (-8.2,-6) rectangle (.1,-2);
\end{scope}
\end{tikzpicture}
\end{center}
   \caption{
   A graphical summary of the results about words of complex Ginibre matrices in the paper. Arrows represent implications. In Section \ref{extension_section} these results are generalized for words of non-Hermitian matrices with i.i.d.\! entries and band matrices, whereas only a (weaker) version of the results circled in red can be extended to words of sparse matrices by our techniques.}
   \label{fig:plan_of_paper}  
\end{figure}
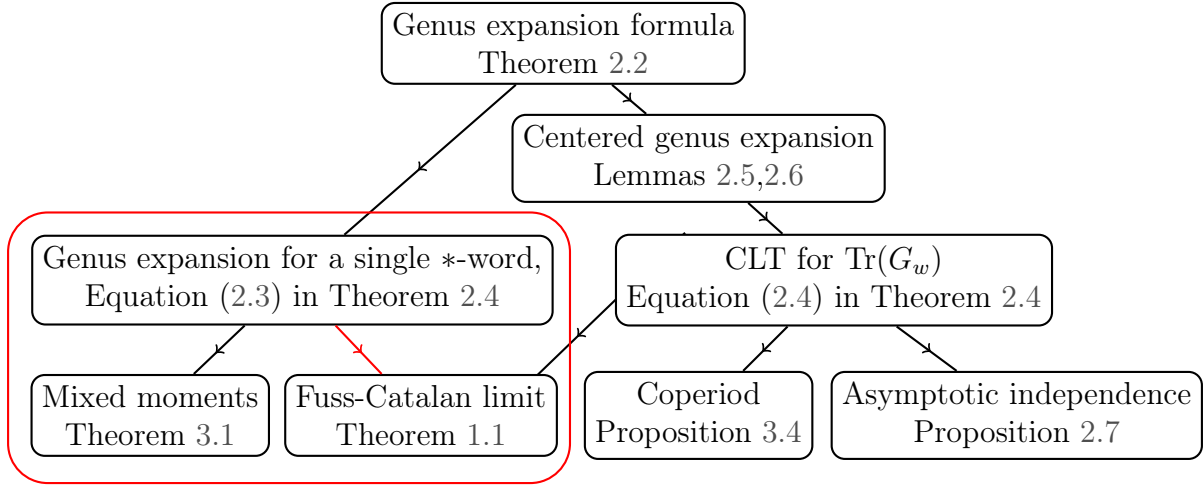

\section{Words of complex Ginibre matrices}\label{SurfacesSection}

\subsection{Notations and conventions}

Complex Ginibre matrices are random matrices with i.i.d.\! standard complex Gaussian coefficients. Throughout this paper $G_1,G_2,...$ are i.i.d complex Ginibre matrices of size $N \times N$, scaled such that every entry has variance $\frac{1}{N}$. \medskip

A $*$-word that involves each $G_i$ and $G_i^*$ the same number of times will be called {\em balanced}. A $*$-word $w$ such that $w=w^*$ will be called {\em $*$-stable}, where $w^*$ is obtained from $w$ by reversing its order and take the conjugate transpose of every letter. To every letter of a $*$-word we associate two parameters: (i) an integer index $j$ associated to $G_j$ and $G_j^*$, and (ii) a \textit{sign} which can take two values: star for $G_j^*$ and non-star for $G_j$. \medskip

The two-dimensional sphere (of genus $0$) will be denoted by $\S$. For every $g \geq 1$, we will denote by $S_g$ the generic compact connected orientable surface of genus $g$ (a torus with $g$ holes). For non-orientable surfaces, which appear in Subsection \ref{polyletters}, we will denote by $N_g$ the connected sum of $g$ real projective planes, such that $N_1 \simeq \Proj$, and $N_2$ is homeomorphic to the Klein bottle.

\subsection{Complex Gaussian Variables and Wick's principle}

A complex centered Gaussian variable $Z \sim \mathscr{N}_{\C}(0, \Sigma)$ can be defined as $X+iY$, where $(X,Y)$ is a real Gaussian vector with covariance matrix $\Sigma$. A case of special interest is when $X$ and $Y$ are independent, i.e., $\Sigma = \frac{\sigma^2}{2} I_2$ for some $\sigma>0$. In such case, the variable $Z$ has the density
$
\frac{1}{\sigma^2 \pi} \exp({-|z|^2/\sigma^2})
$
with respect to the Lebesgue measure on $\C$. We will denote this distribution as $ \mathscr{N}_{\C}(0, \sigma^2) $, and refer to the case $\sigma=1$ as \textit{standard}. In particular, then, $\E Z^2 =0 $ and $\E |Z|^2 =1$. \medskip

A useful feature of the Gaussian distribution, whether real or complex, is Wick's principle, that we now present. Let $S$ be a set of even size $M$. A {\em pairing} $\phi$ of $S$ is a partition of $S$ to subsets $S_1,...,S_{M/2}$ of size 2, and we denote the elements of $S_j$ by $S_j=\{\phi_{j,1},\phi_{j,2}\}$. The set of all pairings of $S$ is denoted by $\Pair(S)$. 

\begin{proposition}[Wick's principle]\label{Wick} Let $(Z_{i})_{i \geq 1}$ be a complex centered Gaussian vector, $Z_{-i}=\overline{Z_i}$, and $(i_k)_{k=1}^M$ is any sequence of non-zero integers in $\Z$, then
$$
\E \lll( \prod_{k=1}^M Z_{i_k} \rr)
=
\sum_{\phi \in \Pair([M])} \prod_{j=1}^{M/2} \E \lll( Z_{i_{\phi_{j,1}}} Z_{i_{\phi_{j,2}}} \rr)
$$
where the sum is taken over every possible pairing of $[M]:=\{1,...,M\}$.

\end{proposition}

Clearly, when dealing with i.i.d.\! standard complex Gaussian variables, only pairings that match each letter $Z_i$ with one of its conjugates have a nonzero contribution. Wick's principle is extremely useful, as it enables one to compute the expectation of any product of Gaussian variables in terms of pairwise covariances. It is used extensively in various areas of theoretical physics and probability theory.

\subsection{Genus expansion for complex Ginibre matrices}
Let $w_1, \dots, w_k$ be $*$-words of complex Ginibre matrices. The goal of this subsection is to derive a genus expansion formula for the expectation $$\E\left[\prod_{j=1}^{k}\Tr(G_{w_j})\right].$$

To each $*$-word $w_j$ we associate an oriented polygonal face $f_{w_j}$ with $|w_j|$ edges which are labelled by $w_j$'s letters. The orientation of the face $f_{w_j}$ agrees with the order of $w_j$'s letters. In addition, we endow the edges with an direction, such that an edge associated to a star-free letter $G_r$ will be oriented by an arrow in the direct way around its face, whereas an edge associated to the star-letter $G_r^*$ will be oriented in the opposite direction. This convention matches the effect of transposition of matrices and helps to visualize which pairings are admissible, in the sense defined below. \medskip

A pairing of the edges of all $k$ faces is called {\em admissible} if every two paired edges are labelled by $G_r$ and $G_r^*$ for some $r\ge 1$. In other words, paired edges should have same index and opposite signs (see Figure \ref{Example1}). We denote the set of all admissible pairings by $\AdPair(w_1,\ldots,w_k)$.

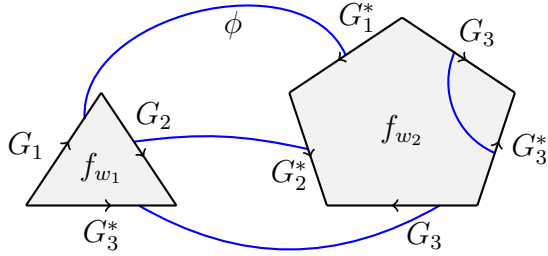
\begin{figure}[ht]
\begin{center}
\begin{tikzpicture}
\begin{scope}[thick,decoration={
    markings,
    mark=at position 0.57 with {\arrow{>}}}
    ] 
\filldraw[fill=black!5!white,draw=white] (-1,-0.5) -- (1,-0.5) -- (0,1) -- (-1,-0.5);
\draw (0,0) node {$f_{w_1}$};
\draw (-1,.3) node {$G_1$};
\draw (.7,.7) node {$G_2$};
\draw (0,-.9) node {$G_3^*$};
\filldraw[fill=black!5!white,draw=white](3,-.5) -- (5,-.5) -- (5.5,1) -- (4,2)-- (2.5,1) -- (3,-.5);
\draw (4,0.5) node {$f_{w_2}$};
\draw (3.4,2) node {$G_1^*$};
\draw (5,1.8) node {$G_3$};
\draw (5.7,0.25) node {$G_3^*$};
\draw (4.3,-.9) node {$G_3$};
\draw (2.5,-.1) node {$G_2^*$};
\draw (1.75,1.9) node {$\phi$};
\draw[color=blue] (5.23,.2) to[bend left=45] (4.7,1.55);
\draw[color=blue] (.5,-.5) to[bend right] (4.5,-.5);
\draw[color=blue] (.445,.35) to[bend left=10] (2.75,.25);
\draw[color=blue] (-.22,.65) to[bend left=80] (3.25,1.5);
\draw[postaction={decorate}] (-1,-0.5) -- (1,-0.5);
\draw[postaction={decorate}] (0,1) -- (1,-0.5);
\draw[postaction={decorate}] (-1,-0.5) -- (0,1);
\draw[postaction={decorate}] (5,-.5) -- (3,-.5);
\draw[postaction={decorate}] (5,-.5) -- (5.5,1);
\draw[postaction={decorate}] (4,2) -- (5.5,1);
\draw[postaction={decorate}] (4,2)-- (2.5,1);
\draw[postaction={decorate}] (2.5,1) -- (3,-.5);
\end{scope}
\end{tikzpicture}
\end{center}
\caption{Faces and edges associated to a product of traces, together with an admissible pairing $\phi$ represented as blue lines. $w_1= G_1 G_2 G_3^*$, $w_2= G_1^* G_3 G_3^* G_3 G_2^*$.}
\label{Example1}
\end{figure}

For an admissible pairing $\phi\in\AdPair(w_1,\ldots,w_k)$, let $S_{\phi}$ be the surface constructed from the faces $f_{w_1},...,f_{w_k}$  by identifying the edges paired by $\phi$ coherently with the orientation. By construction, $S_{\phi}$ is an oriented compact surface. In particular, every connected component of $S_{\phi}$ is a connected surface that is characterized by its genus. We denote by $c(S_\phi)$ the number of connected components of $S_{\phi}$, and $g(S_\phi) = \sum_{i=1}^{c(S_{\phi})} g_i$ the sum of the genera of its components. We also denote by $V(S_\phi),E(S_\phi),F(S_\phi)$ the number of vertices, edges and faces of the graph $\Gamma_{\phi}$ on the surface $S_{\phi}$ obtained by identifying the edges of $f_{w_1},...,f_{w_k}$ according to $\phi$. \medskip

\begin{theorem}[Genus expansion]\label{genusexpansion}
Let $w_1, \dots, w_k$ be $*$-words and $N$ an integer. Then,
$$
\E \left( \prod_{i=1}^k \Tr( G_{w_i}) \right) = \sum_{\phi \in \AdPair(w_1, \dots, w_k)} {N^{2c({S_\phi}) -k - 2 g (S_\phi)}}
$$

\end{theorem}
This follows from Wick's formula in all generality; We give below a direct proof in order to introduce notations and ideas that are needed later in the paper.
\begin{proof}
Let $\mathscr{V}=\{(G_r)_{s,t}~:~r\ge 1,1\le s,t \le N\}$, $m_j$ denote the length of the $*$-word $w_j$ for $j=1,...,k$, $m:=m_1+\cdots + m_k$ and  \[\mathscr{I}=[N]^{m_1}\times[N]^{m_2}\times\cdots \times[N]^{m_k}\] be the set of {\em indexations}. 
For every indexation $I=(i_{j,l})_{1\le j \le k,1\le l\le m_j}\in\mathscr{I}$ we denote the product 
\[
G(I):= \prod_{j=1}^{k}\prod_{l=1}^{m_j} G_{i_{j,l},i_{j,l+1}}^{(j,l)}
\]
where $G^{(j,l)}$ is the $l$-th letter of the $*$-word $w_j$, and we use the circular convention $i_{j,m_j+1}=i_{j,1}$. 

Given an indexation $I$, let $\mathscr{V}(I)$ consist of all the tuples $v=(r,s,t)$, where $r\ge 1$ and $1\le s,t \le N$, such that the Gaussian variable $(G_r)_{s,t}\in \mathscr{V}$ is one of the terms in the product $G(I)$. We call an indexation {\em balanced} if for every $v=(r,s,t)\in\mathscr V(I)$ the variable $(G_r)_{s,t}\in \mathscr{V}$ appears in $G(I)$ the exact same number of times as its conjugate $(G_r^*)_{t,s}$. In such case, we refer to this number as the {\em multiplicity} of $v$ and denote it by $m_v$. Clearly, $\E[G(I)]= 0 $ if $I$ is not balanced, and otherwise, by our normalization of the Gaussian variables, $\E[G(I)]=N^{-\frac{m}2}\prod_{v\in\mathscr{V}(I)} m_v!$. \medskip

Due to the fact that $\prod_{j=1}^{k}\Tr(G_{w_j})=\sum_{I\in\mathscr{I}}G(I)$, we derive the theorem by finding a combinatorial meaning for the sum $\sum_{I\in\mathscr{I}} \prod_{v\in\mathscr{V}(I)} m_v!$. 

Namely, we say that an indexation $I\in\mathscr{I}$ is {\em compatible} with an admissible pairing $\phi\in\AdPair(w_1,...,w_k)$ if for every two paired letters $G^{(j,l)}$ and $G^{(j',l')}$ we have that 
\[
i_{j,l} = i_{j',l'+1} \mbox{ and } i_{j,l+1} = i_{j',l'}.
\]
In other words, the indexation respects the gluing of the edges according to $\phi$. Equivalently, the indexation consistently induces a labeling of the vertices of the surface $S_{\phi}$ where each vertex has a label in $\{1,...,N\}$. In consequence, there are exactly $N^{V(\phi)}$ balanced indexations that are compatible with any admissible pairing $\phi$.
On the other hand, in order to construct an admissible pairing $\phi$ that a given balanced indexation $I$ is compatible with, we need to choose, for every $v=(r,s,t)\in\mathscr{V}(I)$ a matching between the $m_v$ copies of $(G_r)_{s,t}$ and $(G_r^*)_{t,s}$ in $G(I).$
Therefore, 
\[
\sum_{I\in\mathscr{I}_B}\prod_{v\in\mathscr{V}(I)} m_v! 
~=
\sum_{\phi\in\AdPair(w_1,...,w_k)} N^{V(S_{\phi})},
\]
where $\mathscr{I}_B$ denotes the subset of balanced indexations. Consequently,
\[
\E \left( \prod_{i=1}^k \Tr( G_{w_i}) \right) = 
\sum_{I\in\mathscr{I}_B}\E[G(I)]=
\sum_{\phi\in\AdPair(w_1,...,w_k)} N^{V(S_{\phi})-m/2}.
\]

It clearly holds that $E(S_{\phi})=m/2$ and $F(S_{\phi})=k$ for every admissible pairing $\phi$, and the proof is concluded by recalling that  the Euler's characteristic of $S_{\phi}$ is equal to
\begin{equation*}
V(S_{\phi}) - E(S_{\phi}) +F(S_{\phi})= 2 c (S_{\phi}) -2g (S_{\phi}).
 \end{equation*}
 which follows by summing the Euler characteristic $\chi = 2-g$ over all connected components of $S_{\phi}$.
 \end{proof}

\subsubsection{Examples.} Here are two straightforward examples of Theorem \ref{genusexpansion}.

\begin{itemize}
\item \textbf{Moments of $\Tr\left( G \right)$.} Entries of $G$ are independent and $\mathscr{N}_{\C}(0,1/N)$-distributed, hence
$$
Z := \Tr\left( G \right) = \sum_{i=1}^N G_{ii} \sim \mathscr{N}_{\C}(0,1).
$$
Therefore, the mixed moments of $Z$ are
$
\E Z^k \overline{Z}^l = k! \delta_{k,l}.
$
We recover this elementary fact from Theorem \ref{genusexpansion} with $w_1= \dots = w_k = G$ and $w_{k+1}= \dots = w_{k+l} = G^*$. There is no admissible pairing if $k \neq l$, and if  $k=l$ there are exactly $k!$ admissible pairings between the $k$ monogons with direct orientation (associated with the $G$'s) and the $k$ monogons with indirect orientation (associated with the $G^*$'s). The surface $S_\phi$ that any such pairing $\phi$ yields is a disjoint union of $k$ spheres, hence $2c(S_\phi)-2k-g(S_\phi)=0$, and the result follows. \medskip

\item \textbf{Moments of $\Tr\left( GG^* \right)$ and random permutations.} By a direct computation,
$$
X := \Tr\left( G G^* \right) = \sum_{i,j} |G_{ij}|^2 \sim \frac{1}{N} \gamma_{N^2},
$$
where $\gamma_{N^2}$ stands for a Gamma variable with parameter $N^2$. It follows from the definition of the Gamma density that the moments of $X$ are given by the following product:
\begin{equation}
\label{eqn:gammaN2}
\E X^k
= \frac{1}{N^k} \frac{\Gamma(N^2 + k)}{\Gamma(N^2)}
= \frac{1}{N^k}
(N^2+k-1) \cdots (N^2+1) N^2.
\end{equation}
On the other hand, from the topological perspective of Theorem \ref{genusexpansion}, the $k$-th moment of $X$ can be computed by summing over the pairings of $k$ alternate `bi-gons' (faces with two edges labelled by $G$ and $G^*$). The $G^*$-edge of a bi-gon can be paired either with the $G$-edge of the same bi-gon, yielding a sphere component, or with the $G$-edge of another bi-gon, yielding a new bi-gon. The surface obtained in the end consists only of sphere components. This makes the combinatorics quite straight\-forward, as a pairing is equivalent to a permutation $\sigma$ in $\mathfrak{S}_k$ defined as follows : $\sigma(i)=j$ if and only if the $G^*$-edge of the bi-gon number $i$ is paired with the $G$-edge of the bi-gon number $j$. The number of sphere components in the resulting surface is the number of cycles of $\sigma$, that we denote $c(\sigma)$. A pairing that yields $c$ spheres will have contribution $N^{2c - k }$ in the genus expansion, so that we can write:
\begin{equation}\label{cycles}
\E X^k =
\sum_{\sigma \in \mathfrak{S}_k} N^{2 c(\sigma) - k}.
\end{equation}
We note that (\ref{eqn:gammaN2}) is equal to (\ref{cycles}) by the fact that the number of cycles in a random permutation chosen uniformly from $\mathfrak{S}_k$ is distributed like the sum $B_1 + \cdots + B_k$ of independent Bernoulli variables, $B_j$ having parameter $1/j$. 

\end{itemize}

\subsection{Limit theorem for the trace of a $*$-word}\label{CLTsection}
We now state and prove a limit theorem for the random variable $\Tr(G_w)$, where $w$ is any $*$-word. The limit is a complex Gaussian random variable whose parameters depend on combinatorial parameters of the $*$-word that arise by the topological genus expansion.  We set the following definitions.
\begin{definition}
Let $w, w_1,...,w_k$ be  $*$-words.
 \begin{enumerate}
\item  We say that a pairing $\phi\in\AdPair(w_1,...,w_k)$ is {\em spherical } if $S_{\phi}$ is homeomorphic to $\S$.

\item We say that $\phi \in \AdPair(w_1, \dots ,w_k)$ contains an {\em atom} if for some $1\le i \le k$, the edges of the polygonal face $f_{w_i}$ are paired within themselves and the corresponding component in $S_{\phi}$ is homeomorphic to $\S$. In such case we also say that $\phi$ induces an atom on $f_{w_i}$.

\item $\phi$ is called atom-free if it does not contain an atom, and $\AdPair_{AF}(w_1,...,w_k)$ denotes the set of atom-free admissible pairings.

\end{enumerate}
\end{definition}

We denote by $a_w,b_w$ and $c_w$  the number of spherical pairings in $\AdPair(w),\AdPair(w,w^*)$ and $\AdPair(w,w)$ respectively. It is easy to see that $a_w=a_{w^*},b_w=b_{w^*}, c_w=c_{w^*}$ and $b_w \geq 1$. 

Theorem \ref{abcthm} below is a precise statement of the main Theorem \ref{thm:clt*}. A notable consequence is that $b_w \geq c_w$ for any word $w$ -- a fact that is not obvious, but certainly can be established by purely topological considerations.

\begin{theorem}\label{abcthm}
For every $*$-word $w$ there holds 
\begin{equation}\label{first_order}
    \E \lll( \Tr (G_w) \rr) = a_w N + \underset{N \rightarrow \infty}{O} \lll(\frac{1}{N}\rr)
\end{equation}
and 
\begin{equation}\label{second_order}
    \Tr (G_w) - a_w N \distconv X+iY,
\end{equation}
where $X,Y$ are two independent real centered Gaussian variables with variances $\frac{b_w+c_w}{2}$ and $\frac{b_w-c_w}{2}$ respectively.
\end{theorem}

The proof is carried out by the method of moments. In order to apply it we need a centralized version of Theorem \ref{genusexpansion}.

\begin{lemma}
\label{lemma:centralized}
Let $w_1,...,w_k$ be $*$-words and $N$ an integer. Then,
\[
\E \left( \prod_{i=1}^k \lll( \Tr( G_{w_i})-a_{w_i}N \rr) \right) = 
\sum_{\phi\in\AdPair_{AF}(w_1,...,w_k)}N^{2c(S_\phi)-k-2g(S_{\phi})}.
\]
\end{lemma}

\begin{proof}
Let $T\subseteq \{1,...,k\}.$ By Theorem \ref{genusexpansion} we find that $$
\E \left( \prod_{i\in T} \Tr( G_{w_i}) \right) = \sum_{t \geq 0} P_{T,t}N^{|T|-2t},
$$
where $P_{T,t}$ is the number of pairings $\phi\in\AdPair(w_i~:~i\in T)$ for which $c(S_{\phi})-g(S_{\phi})=|T|-t$. The summation is over $t\ge 0$ since $c(S_{\phi})\le |T|$ and $g(S_{\phi})\ge 0$. Note that this sum is actually finite, since there are only finitely many admissible pairings. A simple manipulation shows that
$$
\E \left( \prod_{i=1}^k \lll( \Tr( G_{w_i})-a_{w_i}N \rr) \right) = 
\sum_{t\geq 0} \left( \sum_{T\subseteq \{1,...,k\}}(-1)^{k-|T|}P_{T,t}\prod_{i\notin T}a_{w_i}\right) N^{k-2t}.
$$
The term $P_{T,t}\prod_{i\notin T}a_{w_i}$ counts admissible pairings $\phi$ in $\AdPair(w_1,...,w_k)$ that induce $k-|T|$ spheres --- an atom on $f_{w_i}$ for every $i\notin T$ --- and a surface $S'$ on $\{f_{w_i}~:~i\in T\}$ satisfying $c(S')-g(S')=|T|-t$. In particular, every such $\phi$ satisfies $c(S_{\phi})-g(S_{\phi})=k-t$. Hence, the summation
$$\sum_{T\subseteq \{1,...,k\}}(-1)^{k-|T|}P_{T,t}\prod_{i\notin T}a_{w_i}$$
counts, by inclusion-exclusion, 
the number of pairings $\phi\in\AdPair_{AF}(w_1,...,w_k)$ for which $c(S_{\phi})-g(S_{\phi})=k-t$. This completes the proof of Lemma \ref{lemma:centralized}.

\end{proof}
We call an admissible pairing $\phi$ {\em bi-atomic} if $S_\phi$ is homeomorphic to the disjoint union of $k/2$ copies of $\S$. We denote by $d_{w_1,...,w_k}$ the number of {\em bi-atomic} pairings. Note that $d_{w_1,...,w_k}=0$ if $k$ is odd.

\begin{lemma}\label{lemma:diatomic} For every $\phi\in\AdPair_{AF}(w_1,...,w_k)$ there holds $2c(S_{\phi})-k-2g(S_{\phi})\le 0$. Equality holds if and only if $S_\phi$ is homeomorphic to a disjoint union of $k/2$ copies of $\S$. In particular, we have
$$
\E \left( \prod_{i=1}^k \lll( \Tr( G_{w_i})-a_{w_i}N \rr) \right) 
=d_{w_1,...,w_k} + O\left(\frac 1{N}\right).
$$
\end{lemma}

\begin{proof}
Let $S_1,...,S_l$ be the connected components of $S_{\phi}$. Denote by $|S_i|$ the number of polygonal faces in $S_i$ and its genus by $g_i$. Then,
\begin{equation}\label{eqn:gc2k}
c(S_{\phi})-g(S_{\phi})=\sum_{i=1}^{l}1-g_i \le \sum_{i~:~|S_i|\ge 2} 1-g_i
\le |\{i:|S_i|\ge 2\}| \le k/2.
\end{equation}
The first inequality is by our assumption that $\phi$ is atom-free. Indeed, this implies that if $|S_i|=1$ then $g_i\ge 1$, since otherwise $S_i$ would be an atom. The second inequality is by the non-negativity of the genus, and the third inequality is immediate: the number of components with at least two faces cannot exceed the total number of faces over two. This analysis also shows that equality holds if and only if there are exactly $k/2$ components of genus $0$. The statement of Lemma \ref{lemma:diatomic} follows.
\end{proof}


\begin{proof}[Proof of Theorem \ref{abcthm}]
The fact that $\E \Tr(G_w)=a_w N+O(1/N)$ is a special case of Lemma \ref{lemma:centralized}, but it also follows directly from Theorem \ref{genusexpansion} since $2c(S_{\phi})-1-2g(S_{\phi})\le 1$ for every $\phi\in\AdPair(w)$ and equality holds if and only if $\phi$ is spherical. \medskip 

Let $Z=\Tr(G_w)-a_wN$ and consider the mixed moment $\E(Z^k\overline Z^l)$ for $k,l\ge 0$. It is clear that $\overline Z = \Tr(G_{w^*})$, hence $\E(Z^k\overline Z^l)$ falls under the scope of Lemma \ref{lemma:centralized} by considering $k$ copies of $w$ and $l$ copies of $w^*$. In particular, the main term of $\E(Z^k\overline Z^l)$ equals the number of admissible pairings $\phi$ for which $S_{\phi}$ is homemorphic to $(k+l)/2$ spheres. \medskip

This can be computed by going over all pairings of $k$ $w$-labelled faces and $l$ $w^*$-labelled faces; the contribution of each such pairing to the main term is $b_w^{d}c_w^{(k+l)/2-d}$, where $d$ is the number of $w$-labelled faces that are matched with a $w^*$-labelled face. Indeed, every such pair of faces can construct a $2$-sphere in $b_w$ ways, while pairs of $w$-labelled faces (as well as $w^*$-labelled faces) can construct a $2$-sphere in $c_w$ ways. \medskip

It remains to prove that these are precisely the mixed moments of $Z'=X+iY$ with the expected covariance structure. Indeed, Wick's Theorem yields the exact same summation for the mixed moment of $Z'$ except that $b_w^{d}c_w^{(k+l)/2-d}$ is replaced by 
$$
 \lll( \E Z' \overline{Z'} \rr)^{d}\lll( \E Z'^2 \rr)^{\frac{k-d}{2}} \lll( \E \overline{Z'}^2 \rr)^{\frac{l-d}{2}}.
$$
The proof is concluded by observing that 
$\E Z' \overline{Z'} = \E X^2 + \E Y^2 =b_w$ and $\E Z'^2= \E \overline{Z'}^2 = \E X^2 - \E Y^2 = c_w$. \end{proof}



We say that two $*$-free words are {\em trace-distinct} if they are not equal up to a cyclic permutation of the letters. Genus expansion also allows to prove asymptotical independence properties (referred to as second order freeness in the works of Speicher \& al. \cites{MingoSpeicher1, MingoSpeicher2}). Here we prove asymptotical independence for $*$-free trace-distinct words.

\begin{proposition}\label{independence}
For any set of $*$-free trace-distinct words $w_1, \dots, w_m $, 
$$
\lll( \Tr( G_{w_i}) \rr)_{i=1}^m \distconv (Z_1, \dots, Z_m)
$$
where the limit is a vector of independent complex centered Gaussian variables with variances given by Theorem \ref{abcthm}.
\end{proposition}

\begin{remark}
Clearly, if two $*$-free words $w_1$ and $w_2$ are not trace-distinct then the traces of $G_{w_1}$ and $G_{w_2}$ are equal. Note that Proposition \ref{independence} yields in particular a converse statement. 
\end{remark}

\begin{proof} The proof relies on the moments method, as all we need to prove is that every mixed moment factorizes. For every two sequences $(k_i)_{i=1}^{m}, (l_i)_{i=1}^{m}$ of integers, the mixed moment we want to compute is
$$
\E \prod_{i=1}^m \Tr( G_{w_i})^{k_i} \overline{\Tr( G_{w_i})}^{l_i}.
$$
Let $k:=\sum_i k_i+l_i$. Clearly, $a_{w_i}=0$ since $w_i$ is $*$-free. Therefore, by Lemma \ref{lemma:diatomic}, the leading term of the mixed moment is given by admissible pairings that yield a disjoint union of $k/2$ spheres. Because the words are trace-distinct, $w_i$ can be paired only with $w_i^*$, hence such an admissible pairing exists only if $k_i=l_i$ for every $i$. In such case, the number of these admissible pairings is equal to $\prod_i b_{w_i}^{l_i}l_i!$. Indeed, there are $l_i!$ ways to pair the $w_i$'s with the $w_i^*$'s, and each such pair can create $b_{w_i}$ different spheres. Note that by a similar argument, $\E \Tr( G_{w_i})^{k_i} \overline{\Tr( G_{w_i})}^{l_i}$ vanishes if $k_i\neq l_i$ and is equal to $b_{w_i}^{l_i}l_i!$ otherwise, for every $1\le i\le m$. In conclusion,
$$
\lim_{N \rightarrow \infty} \E \prod_{i=1}^m \Tr( G_{w_i})^{k_i} \overline{\Tr( G_{w_i})}^{l_i}
=
\prod_{i=1}^m \lim_{N \rightarrow \infty} \E \lll( \Tr( G_{w_i})^{k_i} \overline{\Tr( G_{w_i})}^{l_i} \rr).
$$
The result follows.
\end{proof}

\subsubsection{The structure of spherical pairings.}\label{spherebuilding}

In order to use Theorem \ref{abcthm}, one needs to evaluate the parameters $a_w, b_w, c_w$ of a $*$-word $w$. This evaluation boils down to counting spherical pairings of one or two labelled faces. Here we specify the three rules that characterize spherical pairings of one or two faces: (i) no internal crossing, (ii) no external crossing, and (iii) no `bridge'. 

\begin{fact}
An admissible pairing of one or two labelled faces is spherical if and only if conditions (i), (ii) and (iii) presented below are met.
\end{fact}

The proof of this fact will be made clear as we go through each one of these three conditions. Note that the first one appears very commonly in the literature, as many computations are equivalent to calculating the term $a_w$, which simply amounts to counting non-crossing pairings. More generally, conditions on annular non-crossing permutations and partitions are presented in \cite{MingoNica2004}. \\

 \textbf{(i) Internal crossings:} If the edges $e_1,e_2,e_3,e_4$ are situated around a given face in that order, then pairing $e_1$ with $e_3$ and $e_2$ with $e_4$ does not yield a sphere. Indeed, as illustrated on Figure \ref{Internal_crossing}, there will be two paths that intersect only once in a transverse way on the resulting surface of such a pairing, which is impossible on a sphere. 
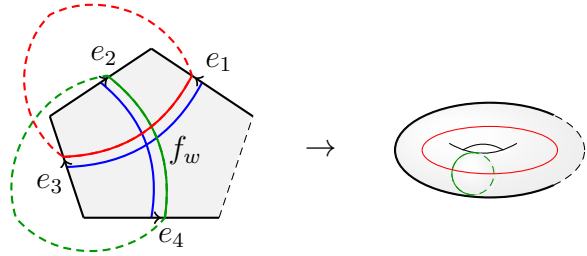
\begin{figure}[h!]
\begin{center}
\begin{tikzpicture}[scale=.9]
\begin{scope}[thick,decoration={
    markings,
    mark=at position 0.57 with {\arrow{>}}}
    ] 
\filldraw[fill=black!5!white,draw=white](-2,-1) -- (0,-1) -- (.5,.5) -- (-1,1.5)-- (-2.5,.5) -- (-2,-1);
  \draw[color=blue] (-1,-1) to[bend right] (-1.75,1);
  \draw[color=blue] (-2.25,-.25) to[bend right] (-.25,1);
\draw[postaction={decorate}] (.5,.5) -- (-1,1.5);
\draw[postaction={decorate}] (-2.5,.5) -- (-1,1.5);
\draw[postaction={decorate}] (-2,-1) -- (-2.5,.5);
\draw[postaction={decorate}] (-2,-1) -- (0,-1);
\end{scope}
\draw (1.5,0) node {$\rightarrow$};
\draw[densely dashed] (0,-1) -- (.5,.5);
\draw (-.5,0) node {$f_w$};
\draw (0,1.3) node {$e_1$};
\draw (-1.7,1.4) node {$e_2$};
\draw (-2.5,-.5) node {$e_3$};
\draw (-.7,-1.3) node {$e_4$};
 \draw[color=black!40!green,thick,densely dashed] (-.8,-1) to[bend right=30] (-1.65,1.1) to[bend right=50] (-3,-1) to[bend right=50] (-.8,-1) ; 
  \draw[color=black!40!green,thick] (-.8,-1) to[bend right=30] (-1.65,1.1); 
 \draw[color=red,thick,densely dashed] (-2.3,-.1) to[bend right=30] (-.4,1.1) to[bend right=50] (-2.5,2) to[bend right=50] (-2.3,-.1) ; 
 \draw[color=red,thick] (-2.3,-.1) to[bend right=30] (-.4,1.1) ; 
  \draw[thick] (4,0) ellipse (1.4cm and .7cm);
  \filldraw[fill=white,opacity=1,draw=white] (5.5,0) circle (.75cm);
  \draw[densely dashed] (4,0) ellipse (1.4cm and .7cm);
 \shade[ball color=black!10!white, opacity=.1] (4,0) ellipse (1.4cm and .7cm);
   \filldraw[color=white,opacity=1] (3.7,0) arc (180:360:.3cm and 0.05cm) -- (4,.09) -- cycle;
  \draw (3.6,0) to[bend left] (4.2,0);
  \draw (3.4,.1) to[bend right] (4.4,.1);
  \draw[color=black!40!green] (3.75,-.35) ellipse (.31cm and .31cm);  
    \filldraw[color=white] (3.8,-.35) ellipse (.3cm and .3cm);  
  \shade[ball color=black!10!white, opacity=.1] (3.8,-.35) ellipse (.3cm and .3cm);  
  \draw[densely dashed,color=black!40!green] (3.75,-.35) ellipse (.31cm and .31cm);  
  \draw[color=red] (4,0) ellipse (1cm and .35cm);  
\end{tikzpicture}
\end{center}
\caption{Pairings that cross internally do not yield spheres. Such a pairing is represented in blue, and two resulting paths in green and red. Only one crossing happens on the face, leading to one transverse intersection on the surface.}
\label{Internal_crossing}
\end{figure}

\textbf{(ii) External crossing}: If the edges $e_1,e_2,e_3$ and $f_3, f_2, f_1$ are situated around two different faces in this order, then pairing $e_1$ with $f_1$, $e_2$ with $f_3$ and $e_3$ with $f_2$ never yields a sphere. This can either be seen by drawing \textit{ad hoc} paths as previously (see Figure \ref{External_Crossing}), or by considering the face formed after pairing $e_1$ with $f_1$, and then apply the internal crossings condition. 
    
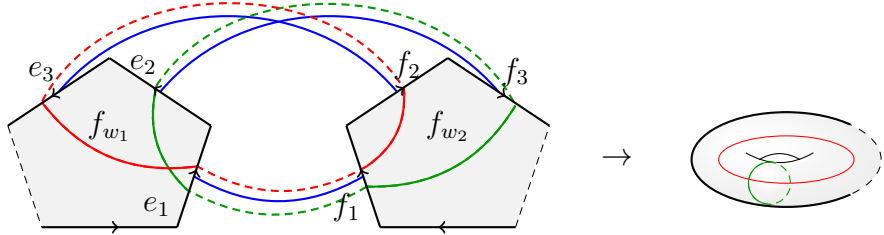
\begin{figure}[h!]
\begin{center}
\begin{tikzpicture}[scale=.9]
\begin{scope}[thick,decoration={
    markings,
    mark=at position 0.57 with {\arrow{>}}}
    ] 
\filldraw[fill=black!5!white,draw=white](-2,-1) -- (0,-1) -- (.5,.5) -- (-1,1.5)-- (-2.5,.5) -- (-2,-1);
\filldraw[fill=black!5!white,draw=white](-7,-1) -- (-5,-1) -- (-4.5,.5) -- (-6,1.5)-- (-7.5,.5) -- (-7,-1);
  \draw[color=blue] (-.25,1) to[bend right=50] (-5.25,1);
  \draw[color=blue] (-1.75,1) to[bend right=50] (-6.75,1);
  \draw[color=blue] (-2.25,-.25) to[bend left] (-4.75,-.25);
  \draw[color=black!40!green,densely dashed] (-2.2,-.4) to[bend left] (-4.85,-.45) to[bend left] (-5.3,1.1) to[bend left=60] (0,.8) to[bend left] (-2.2,-.4) ; 
   \draw[color=black!40!green] (-4.85,-.45) to[bend left] (-5.3,1.1) ;
   \draw[color=black!40!green] (0,.8) to[bend left] (-2.2,-.4) ; 
 \draw[color=red,densely dashed] (-2.2,-.1) to[bend left] (-4.7,-.1) to[bend left] (-7,.85) to[bend left=60] (-1.65,1.05) to[bend left] (-2.2,-.1) ; 
 \draw[color=red] (-4.7,-.1) to[bend left] (-7,.85) ; 
 \draw[color=red] (-1.65,1.05) to[bend left] (-2.2,-.1) ; 
\draw[postaction={decorate}] (-1,1.5) -- (.5,.5);
\draw[postaction={decorate}] (-2.5,.5) -- (-1,1.5);
\draw[postaction={decorate}] (-2,-1) -- (-2.5,.5);
\draw[postaction={decorate}] (0,-1) -- (-2,-1);
\draw[postaction={decorate}] (-4.5,.5) -- (-6,1.5);
\draw[postaction={decorate}] (-6,1.5) -- (-7.5,.5);
\draw[postaction={decorate}] (-5,-1) -- (-4.5,.5);
\draw[postaction={decorate}] (-7,-1) -- (-5,-1);
\end{scope}
\draw[densely dashed] (0,-1) -- (.5,.5);
\draw[densely dashed]  (-7.5,.5) -- (-7,-1);
\draw (-6,0.5) node {$f_{w_1}$};
\draw (-1,0.5) node {$f_{w_2}$};
\draw (-7,1.3) node {$e_3$};
\draw (-5.5,1.4) node {$e_2$};
\draw (-5.3,-.7) node {$e_1$};
\draw (0,1.3) node {$f_3$};
\draw (-1.6,1.4) node {$f_2$};
\draw (-2.5,-.7) node {$f_1$};
\draw (1.5,0) node {$\rightarrow$};
  \draw[thick] (4,0) ellipse (1.4cm and .7cm);
  \filldraw[fill=white,opacity=1,draw=white] (5.5,0) circle (.75cm);
  \draw[dashed] (4,0) ellipse (1.4cm and .7cm);
 \shade[ball color=black!10!white, opacity=.1] (4,0) ellipse (1.4cm and .7cm);
   \filldraw[color=white,opacity=1] (3.7,0) arc (180:360:.3cm and 0.05cm) -- (4,.09) -- cycle;
  \draw (3.6,0) to[bend left] (4.2,0);
  \draw (3.4,.1) to[bend right] (4.4,.1);
  \draw[color=black!40!green] (3.75,-.35) ellipse (.31cm and .31cm);  
    \filldraw[color=white] (3.8,-.35) ellipse (.3cm and .3cm);  
  \shade[ball color=black!10!white, opacity=.1] (3.8,-.35) ellipse (.3cm and .3cm);
  \draw[densely dashed,color=black!40!green] (3.75,-.35) ellipse (.31cm and .31cm);
  \draw[color=red] (4,0) ellipse (1cm and .35cm);  
\end{tikzpicture}
\end{center}
\caption{Pairings that cross externally do not yield spheres.}
\label{External_Crossing}
\end{figure}
\medskip
 
\textbf{(iii) Bridge:} If the edges $e_1,e_2,e_3,e_4$ are situated around a face $w_1$ in this order, and $e_2$ is paired with $e_4$, then pairing both $e_1$ and $e_3$ to two edges $f_1$ and $f_2$ on another face $w_2$ cannot yield a sphere. This can either be seen by drawing \textit{ad hoc} paths as previously (see Figure \ref{Bridge_Crossing}), or by considering the face formed after pairing $e_1$ with $f_1$, and then apply the condition that there should be no internal crossing on that face. This condition generalizes to any number of faces by stating that their should be no possible path linking (`bridging') the two sides of $w_1$.
    
    \begin{figure}[h!]
\begin{center}
\begin{tikzpicture}[scale=.9]
\begin{scope}[thick,decoration={
    markings,
    mark=at position 0.57 with {\arrow{>}}}
    ] 
\filldraw[fill=black!5!white,draw=white](-2,-1) -- (0,-1) -- (.5,.5) -- (-1,1.5)-- (-2.5,.5) -- (-2,-1);
\filldraw[fill=black!5!white,draw=white](-7,-1) -- (-5,-1) -- (-4.5,.5) -- (-6,1.5)-- (-7.5,.5) -- (-7,-1);
  \draw[color=blue] (-6,-1) to[bend right=40] (-1,-1);
  \draw[color=blue] (-1.75,1) to[bend right=50] (-5.25,1);
 \draw[color=blue] (-4.75,-.25) to[bend left] (-6.75,1);
 \draw[color=black!40!green,densely dashed] (-6.85,.9) to[bend right=40] (-4.8,-.37) to[bend right=40] (-4.2,1) to[bend right] (-5.4,2.1) to[bend right=50] (-6.85,.9) ;
  \draw[color=black!40!green] (-6.85,.9) to[bend right=40] (-4.8,-.37) ;
 \draw[color=red,densely dashed] (-5.35,1.1) to[bend left=50] (-1.65,1.1) to[bend left=20] (-.9,-1) to[bend left=50] (-6.1,-1) to[bend left=20] (-5.35,1.1) ; 
  \draw[color=red] (-1.65,1.1) to[bend left=20] (-.9,-1) ; 
   \draw[color=red] (-6.1,-1) to[bend left=20] (-5.35,1.1) ; 
\draw[postaction={decorate}] (-1,1.5) -- (.5,.5);
\draw[postaction={decorate}] (-2.5,.5) -- (-1,1.5);
\draw[postaction={decorate}] (-2,-1) -- (-2.5,.5);
\draw[postaction={decorate}] (0,-1) -- (-2,-1);
\draw[postaction={decorate}] (-4.5,.5) -- (-6,1.5);
\draw[postaction={decorate}] (-7.5,.5) -- (-6,1.5);
\draw[postaction={decorate}] (-5,-1) -- (-4.5,.5);
\draw[postaction={decorate}] (-7,-1) -- (-5,-1);
\end{scope}
\draw[densely dashed] (0,-1) -- (.5,.5);
\draw[densely dashed]  (-7.5,.5) -- (-7,-1);
\draw (-5.3,0.4) node {$f_{w_1}$};
\draw (-.3,0.4) node {$f_{w_2}$};
\draw (-7,1.3) node {$e_4$};
\draw (-5.5,1.4) node {$e_3$};
\draw (-4.6,-.7) node {$e_2$};
\draw (-6,-1.5) node {$e_1$};
\draw (-1.6,1.5) node {$f_2$};
\draw (-1,-1.5) node {$f_1$};
\draw (1.5,0) node {$\rightarrow$};
  \draw[thick] (4,0) ellipse (1.4cm and .7cm);
  \filldraw[fill=white,opacity=1,draw=white] (5.5,0) circle (.75cm);
  \draw[dashed] (4,0) ellipse (1.4cm and .7cm);
 \shade[ball color=black!10!white, opacity=.1] (4,0) ellipse (1.4cm and .7cm);
   \filldraw[color=white,opacity=1] (3.7,0) arc (180:360:.3cm and 0.05cm) -- (4,.09) -- cycle;
  \draw (3.6,0) to[bend left] (4.2,0);
  \draw (3.4,.1) to[bend right] (4.4,.1);
  \draw[color=black!40!green] (3.75,-.35) ellipse (.31cm and .31cm);  
    \filldraw[color=white] (3.8,-.35) ellipse (.3cm and .3cm);  
  \shade[ball color=black!10!white, opacity=.1] (3.8,-.35) ellipse (.3cm and .3cm);  
  \draw[densely dashed,color=black!40!green] (3.75,-.35) ellipse (.31cm and .31cm);  
  \draw[color=red] (4,0) ellipse (1cm and .35cm);  
\end{tikzpicture}
\end{center}
\vspace{-.3cm}
\caption{Pairings that include a bridge do not yield spheres.}
\label{Bridge_Crossing}
\end{figure}
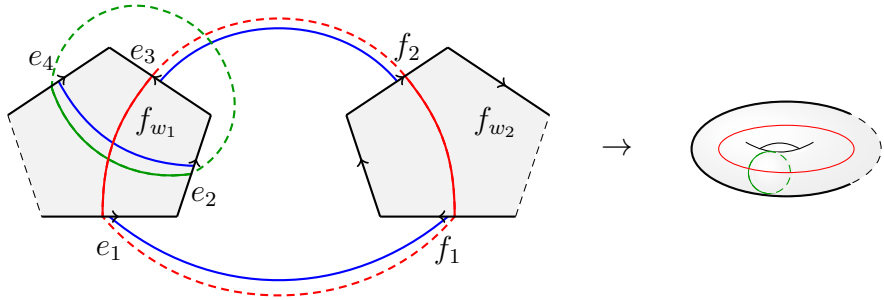

A straightforward induction on the number of available edges shows that any surface that is built from one or two faces respecting rules (i), (ii), (iii) is a sphere.

\subsubsection{Examples.}
We present below a few examples and illustrations of Theorem \ref{abcthm}.

\begin{itemize}

\item \textbf{$*$-Free words.} If $w$ is $*$-free, then there is no admissible pairing of $w$ onto itself or a copy of itself, hence $a_w = c_w =0$. However, $b_w \geq 1$, since at least pairing each letter of $w$ to its conjugate transpose in $w^*$ yields a sphere. Therefore, the limit distribution of $\Tr(G_w)$ is $\mathscr{N}_\C(0,b_w)$.

\item \textbf{$*$-Stable words.} Recall that a $*$-word is said to be $*$-stable if $w=w^*$. This is equivalent to $G_w$ being the covariance matrix $G_w=G_{w_1} G_{w_1}^*$ of a shorter $*$-word $w_1$. If $w$ is $*$-stable, then $b_w=c_w$ and the limit distribution of $\Tr G_w - a_w N$ is $\mathscr{N}_{\R}(0,b_w)$. Moreover, if $w=w_1w_1^*$, where $w_1$ is a $*$-free word, then $a_w=1$. Note that these facts also hold if $w$ and $w^*$ only differ by a cyclic permutation, so that $\Tr G_w = \Tr G_{w^*}$.

\item \textbf{Products of covariance matrices.} We can compute explicitly the parameters $a_{w_m},b_{w_m},c_{w_m}$ of the $*$-word $w_m = G_1 G_1^* \dots G_m G_m^*$, and find
$$
a_{w_m}=1, \quad b_{w_m}=2^m-1, \quad c_{w_m}=\frac{m(m+1)}{2}.
$$
Indeed $a_{w_m}=1$ since there is only one admissible pairing of $w_m$ and it is spherical. In addition, it is easy to see that no admissible pairing of $w_m$ and $w_m^*$ contains neither an internal crossing, an external crossing nor a bridge. Therefore, out of the $2^m$ admissible pairings, the only one that is not spherical yields two spheres, hence $b_{w_m}=2^m-1.$ Finally, no admissible pairing of two copies of $w_m$ has neither an internal crossing nor a bridge. For every $1\le i \le m$, each of the two copies of $w_m$ contains one edge that is labelled $G_i$ and one that is labelled $G_i^*$. These edge are paired either internally or externally, and an external crossing is created if and only if there are more than two external pairings. Hence $c_{w_m}=m+\binom m2.$

In particular, if $m=1$ or $m=2$, then the centered variable $ \Tr (G_w) - N$ converges to a real Gaussian of variance $1$ and $3$ respectively, whereas for $m \geq 3$ it converges to a complex Gaussian with non-standard covariance structure (see Figure \ref{Experiment1}).
\vspace{.1in}
\begin{figure}[ht]
    \centering
   \includegraphics[scale=.7]{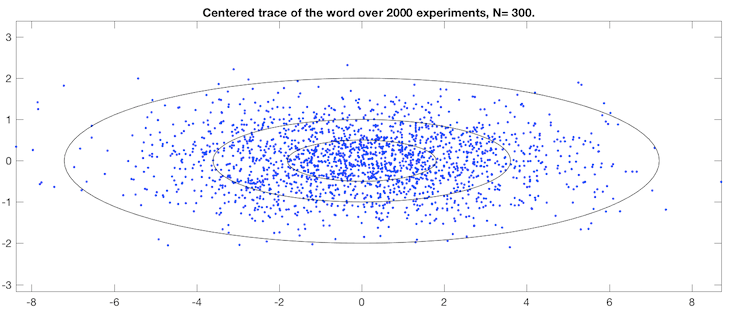}
    \caption{Centered trace of the $*$-word $G_1 G_1^* G_2 G_2^* G_3 G_3^*$ for $N=300$, over $2000$ experiments. According to Theorem \ref{abcthm}, this variable converges to a complex centered Gaussian whose covariance matrix is diagonal with entries $13/2$ and $1/2$. Ellipses of axis ratio $\sqrt{13}$ are represented on the picture.}
    \label{Experiment1}
\end{figure}


\item \textbf{Parameters of the $*$-word $w= G^a G^{*b} $.}

It turns out that the behaviour of $\Tr G^a G^{*b}$ is somewhat different, depending whether $a=b$ or $a<b$. If $a=b$, then a quick check ensures that $a_w=1$, and by symmetry it is obvious that $b_w=c_w$. To compute $b_w$, we observe that every pairing of two copies of $w$ has the following property. If there are $1\le i<j<k \le a$, such that the $j$-th edge of a face is paired internally, and both its $i$-th edge and its $k$-th edge are paired externally, then a bridge is created and the pairing is not spherical. In addition, the $i$-th edge of a face must be paired with the $(a+i)$-th edge of its face or of the other face, in order to avoid internal and external crossing. 

Therefore, in order to create a spherical pairing we select $1\le c_1,c_2,c_3,c_4\le a$ such that $1\le c_1+c_2=c_3+c_4<a$. We pair the first $c_1$ edges and the last $c_2$ edges of the first face internally, the first $c_3$ edges and the last $c_4$ edges of the second face internally, and the remaining edges are paired externally, as shown on Figure \ref{Example3}. Note that because $c=c_1+c_2<a$, some edges will be paired externally. There are $(c+1)^2$ such choices, hence
$$
b_w = \sum_{c=0}^{a-1} (c+1)^2 = S_2 (a)
$$
where $S_2(a)=\frac{1}{6} a(a+1)(2a+1).$ Hence,
$$
\Tr G^a G^{*a} - N \xrightarrow[N \rightarrow \infty]{d} \mathscr{N}_{\R} (0, S_2(a) ).
$$
In particular, when $a=1$ this yields the central limit theorem of $\Tr\left( G G^* \right)$, which is a sum of i.i.d.\! variables.

If $a< b$, then obviously, $a_w=c_w=0$, and a similar argument leads to $b_w=S_2(a+1)$. Therefore,
$$
\Tr G^a G^{*b} \xrightarrow[N \rightarrow \infty]{d} \mathscr{N}_{\C} (0, S_2(a+1) ).
$$
\vspace{.1in}
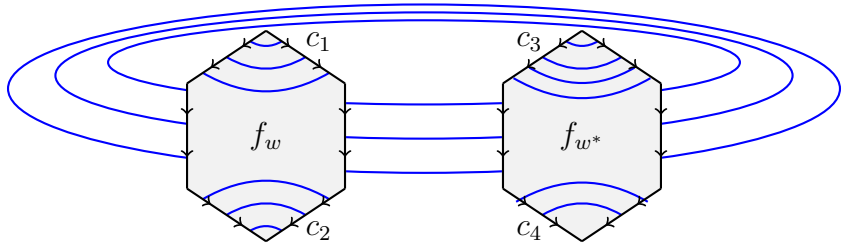
\begin{figure}[ht]
\begin{center}
\begin{tikzpicture}[scale=.7]
\draw[color=blue,thick] (0,1.4) ellipse (6cm and .8cm);
\draw[color=blue,thick] (0,1.15) ellipse (7cm and 1.2cm);
\draw[color=blue,thick] (0,.9) ellipse (7.9cm and 1.6cm);
\begin{scope}[thick,decoration={
    markings,
    mark=at position 0.3 with {\arrow{>}},
    mark=at position 0.7 with {\arrow{>}}}
    ] 
\filldraw[fill=black!5!white,draw=white] (-4.5,-1) -- (-4.5,1) -- (-3,2) -- (-1.5,1) -- (-1.5,-1) -- (-3,-2) -- (-4.5,-1);
\draw (-3,0) node {$f_w$};
\filldraw[fill=black!5!white,draw=white] (4.5,-1) -- (4.5,1) -- (3,2) -- (1.5,1) -- (1.5,-1) -- (3,-2) -- (4.5,-1);
\draw (3,0) node {$f_{w^*}$};
\draw (-2,1.8) node {$c_1$};
\draw (-2,-1.8) node {$c_2$};
\draw (2,1.8) node {$c_3$};
\draw (2,-1.8) node {$c_4$};
\draw[color=blue] (-3.3,1.8) to[bend right] (-2.7,1.8);
\draw[color=blue] (-3.75,1.5) to[bend right] (-2.25,1.5);
\draw[color=blue] (-4.2,1.2) to[bend right] (-1.8,1.2);
\draw[color=blue] (-3.3,-1.8) to[bend left] (-2.7,-1.8);
\draw[color=blue] (-3.75,-1.5) to[bend left] (-2.25,-1.5);
\draw[color=blue] (-4.2,-1.2) to[bend left] (-1.8,-1.2);
\draw[color=blue] (3.3,1.8) to[bend left] (2.7,1.8);
\draw[color=blue] (3.75,1.5) to[bend left] (2.25,1.5);
\draw[color=blue] (4,1.3) to[bend left] (2,1.3);
\draw[color=blue] (4.3,1.1) to[bend left] (1.7,1.1);
\draw[color=blue] (3.75,-1.5) to[bend right] (2.25,-1.5);
\draw[color=blue] (4.25,-1.25) to[bend right] (1.75,-1.25);
\draw[postaction={decorate}] (-3,2) -- (-1.5,1);
\draw[postaction={decorate}] (-1.5,1) -- (-1.5,-1);
\draw[postaction={decorate}] (-1.5,-1) -- (-3,-2);
\draw[postaction={decorate}] (-3,2) -- (-4.5,1);
\draw[postaction={decorate}] (-4.5,1) -- (-4.5,-1);
\draw[postaction={decorate}] (-4.5,-1) -- (-3,-2);
\draw[postaction={decorate}] (3,2) -- (1.5,1);
\draw[postaction={decorate}] (1.5,1) -- (1.5,-1);
\draw[postaction={decorate}] (1.5,-1) -- (3,-2);
\draw[postaction={decorate}] (3,2) -- (4.5,1);
\draw[postaction={decorate}] (4.5,1) -- (4.5,-1);
\draw[postaction={decorate}] (4.5,-1) -- (3,-2);
\end{scope}
\end{tikzpicture}
\end{center}
\caption{Spherical pairings of $w$ and $w^*$, for the $*$-word $w=G^a G^{*a}$.}
\label{Example3}
\end{figure}

\item \textbf{ Parameters of the $*$-word $w=(G_1 G_2 \dots G_a) (G_1 G_2 \dots G_b)^*$.}
Again we distinguish between the case $a=b$ and $a<b$. If $a=b$ then $a_w =1$ and $b_w=c_w$. The computation for $b_w$ can be done in the same way as before. Namely, a spherical pairing is obtained by choosing a contiguous segment of edges that are paired externally. However, since this $*$-word is composed of different letters, this segment must start and end in the same place in both $*$-words. Therefore, 
$$
b_w = \sum_{c=0}^{a-1} c+1 = S_1(a)
$$
where $S_1 (n) := \frac{1}{2} a (a+1)$. On the other hand, if $a<b$ then the edges that are paired internally form a prefix of length at most $a$. Hence,
$$
a_w=0, \quad c_w=0, \quad b_w = a+1.
$$
\end{itemize}
\section{Applications}\label{Applications}

This section is devoted to deriving the consequences of Theorem \ref{abcthm} announced in the introduction, concerning statistics of eigenvalues and singular values of $*$-free words -- that is, expressions involving i.i.d.\! Ginibre matrices without their conjugate transposes.

\subsection{Singular values of words, first order : Fuss-Catalan}
\label{subsec:FC}
We prove here Theorem \ref{thm:mainFC} concerning the limit empirical distribution of the singular values of a complex Ginibre word, and show that it only depends on the length of the word.

\subsubsection{Fuss-Catalan Distributions.}

We will consider the following generalization of Catalan numbers, called Fuss-Catalan numbers of parameter $s \in \N$ :
$$
\forall n \geq 0, \quad
\FC_s (n) = \frac{1}{ sn +1 } \binom{ sn +1 }{n}.
$$
They can be defined in an equivalent way by their initial terms and the recurrence formula
\begin{equation}\label{FCrec}
\FC_s(0)=\FC_{s}(1)=1, \qquad
\FC_s (n+1) = \sum_{k_1 + \dots + k_s = n} \prod_{j=1}^s \FC_s (k_j) 
\end{equation}
where the sum is taken over every possible $s$-uple of non-negative terms that sums up to $n$. This is equivalent to the equation
$$
F_s = 1 + z F_s^s
$$
where $F_s$ is the generating function
$$
F_s(X) : = \sum_{n \geq 0} \FC_s (n) X^n.
$$
Fuss-Catalan numbers, as well as a further generalization to two parameters known as Raney numbers, have been extensively studied. Notably, they appear in \cite{KempSpeicher} in the context of the study of free circular and $\mathscr{R}$-diagonal operators; the diagrams that are considered are equivalent to the ones we obtain in the case of a power $G^m$. \medskip

Importantly, Fuss-Catalan numbers are the moments of a distribution on $\R_+$ that we will refer to as $\rho_{\FC}^s$.

$$
\int_{\R_+} x^k \rho_{\FC}^s(x) \dd x = \FC_{s}(k).
$$

For $s=2$, this gives the Catalan numbers, $\FC_2(n) = C_n$, and the corresponding distribution is the Marchenko-Pastur distribution with shape parameter $\rho=1$, which can also be seen as the quarter circular distribution after a quadratic change of variable:
$$
\rho_{\FC}^2(x) = \frac{1}{2 \pi \sqrt{x}} \sqrt{4-x}.
$$

For $s \geq 3$, the Fuss-Catalan distributions can be expressed as a Meijer G-function with explicit family of parameters, or a combination of hypergeometric functions (see \cite{Penson}). We will not need to rely on any other parameter than the moments in the present paper. A relevant remark is that $\rho_{\FC}^s$ is in fact supported on the interval $[0, \frac{(s+1)^{s+1}}{s^s}]$. It is a well known general fact that the moments method can be applied to measures with compact support. \medskip

\begin{proof}[Proof of Theorem \ref{thm:mainFC}.]
We rely on the moments method, which is to say that convergence of all moments to Fuss-Catalan numbers implies the weak convergence of the measure to the corresponding Fuss-Catalan distribution.

The empirical measure is a random object, for which we claim that weak convergence holds almost surely. This follows from showing that the convergence of the moments occurs almost surely, that is derived directly from Theorem \ref{abcthm}. Indeed, since $(G_wG_{w}^*)^k=G_{(ww^*)^k}$, we find that
$$
\mathrm{Var} \left( \Tr\left( (G_w G_{w}^*)^k \right) - a_{(ww^*)^k} N \right) = \OO(1),
$$
as this variance converges to some finite limit. This implies that
$$
\mathrm{Var} \left( \frac{1}{N} \Tr\left( (G_w G_{w}^*)^k \right) - a_{(ww^*)^k} \right) = \OO(N^{-2}),
$$
and a classical Borel-Cantelli argument allows to conclude that the $k$-th moment of the empirical distribution converges almost surely to $a_{(ww^*)^k}$, the number of spherical pairings of $(w w^*)^k$, which we now compute. \\

Let us first consider the case of the $*$-free word $w=G^m$, and define:
$$
D_k = a_{(w w^*)^k} = \lim \frac{1}{N} \Tr\left( (G^mG^{m*})^k \right).
$$

We need to count spherical pairings of a face with $2km$ edges, that are organized in $k$ groups of $m$ star-free edges interlaced with $k$ groups of $m$ star edges (see Figure \ref{ex3FussCatalan}). Since there is only one face, $a_{(ww^*)^k}$ is equal to the number of admissible pairings without an internal crossing (see subsection \ref{spherebuilding}). This condition has an obvious geometric meaning. Namely, that it is possible to connect all the paired edges by non-crossing lines. Every such line splits the face into two {\em sides}, and the number of star and star-free edges on each side must be equal. As a consequence, a star edge in position $j$ in its group of $m$ star edges can only be paired with a star-free edge in position $m+1-j$ in its group.

\begin{figure}[ht]
\begin{center}
\begin{tikzpicture}[scale=.9]
 \filldraw[densely dashed, fill= black!10!white, draw=black] (-2,3.6) -- (-2.3,3.3) -- (-2.3,1.3) -- (-1,0) -- (1,0) -- (2.3,1.3) -- (2.3,3.3) -- (2,3.6) -- (-2,3.6);
 \draw[thick] (-2,3.6) -- (-2.3,3.3) -- (-2.3,1.3) -- (-1,0) -- (1,0) -- (2.3,1.3) -- (2.3,3.3) -- (2,3.6) ;
 \draw (0,2) node {$f_{(ww^*)^k}$};
     \draw[blue] (-.5,0) to[bend left=45] (2,1);
    \draw[red] (-1.3,0.3) to[bend right=45] (-2.3,1.8);
    \filldraw[fill=white,draw=black] (-2.3-.2,1.8-.2) rectangle (-2.3+.2,1.8+.2); 
    \filldraw[fill=white,draw=black] (-2.3-.2,2.3-.2) rectangle (-2.3+.2,2.3+.2);
    \filldraw[fill=white,draw=black] (-2.3-.2,2.8-.2) rectangle (-2.3+.2,2.8+.2); 
    \filldraw[fill=white,draw=black] (-1.3,0.3) circle (.2cm); 
    \filldraw[fill=white,draw=black] (-1.65,0.65) circle (.2cm); 
    \filldraw[fill=white,draw=black] (-2,1) circle (.2cm); 
    \filldraw[fill=white,draw=black] (-0.5-.2,-0.2) rectangle (-0.5+.2,.2); 
    \filldraw[fill=white,draw=black] (0-.2,-.2) rectangle (0+.2,.2); \filldraw[fill=white,draw=black] (.5-.2,-.2) rectangle (.5+.2,.2); 
    \filldraw[fill=white,draw=black] (1.3,0.3) circle (.2cm); 
    \filldraw[fill=white,draw=black] (1.65,.65) circle (.2cm);
    \filldraw[fill=white,draw=black] (2,1) circle (.2cm); 
        \filldraw[fill=white,draw=black] (2.3-.2,1.8-.2) rectangle (2.3+.2,1.8+.2); 
    \filldraw[fill=white,draw=black] (2.3-.2,2.3-.2) rectangle (2.3+.2,2.3+.2);
    \filldraw[fill=white,draw=black] (2.3-.2,2.8-.2) rectangle (2.3+.2,2.8+.2); 
\end{tikzpicture}
\end{center}
\caption{Part of the diagram obtained for $m=3$. We indicate star-free edges by circles and star-edges by squares. The blue line can be a part of an admissible spherical pairing, but the red line cannot. }
\label{ex3FussCatalan}
\end{figure}
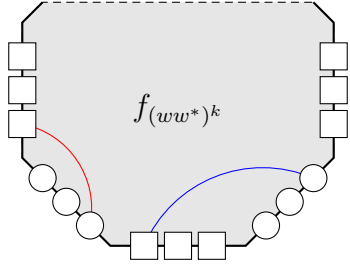

Thus, it is clear by a straightforward recurrence over non-crossing diagrams that
$$D_0=1, \quad D_1=1, \quad D_2=m, \quad D_{N+1} = \sum_{j_1, \dots, j_m} \prod_{i=1}^m D_{j_i}
$$
where the sum is taken over all possible $m+1$-tuples such that $\sum_{i=1}^m j_i=N$. This characterizes the Fuss-Catalan numbers with parameter $m+1$ and concludes the proof for $w=G^m$. \medskip

We observe that the same argument works for any $*$-free word $w= G_{i_1} \dots G_{i_m}$. Strictly speaking, there is now an additional condition that every edge must be paired with an edge of the same index. However, since the absence of internal crossing enforces  a star-free edge in position $j$ in its group to be paired with a star edge in position $m+1-j$ in its group, this additional condition is redundant. Therefore, for any $*$-free word $w$ of length $m$, the moments of $\mu_{ww^*}$ converge to the Fuss-Catalan numbers with parameter $m+1$.
\end{proof}

Convergence of the empirical measure of squared singular values to a Fuss-Catalan distribution was known in the case of powers (see \cite{Alexeev}) and products of independent Ginibre matrices (see \cites{Penson, LambertCLT}). The above theorem extends it to any $*$-free word, showing that only the length of the word matters at first order. The technique can also be applied to the real Ginibre ensemble.

\subsection{Mixed moments of $*$-free words}
For a single complex Ginibre matrix $G$ and any $2k$-tuple of integers $(a_1, b_1, \dots, a_k, b_k)$ we define:
\begin{equation}\label{def_mixed_moments}
\mathscr{M}_{\Gin}(a_1, b_1, \dots, a_k, b_k) : = \lim_{N \rightarrow \infty} \frac{1}{N} \E \lll( \Tr G^{a_1} G^{* b_1} G^{a_2} G^{* b_2} \dots G^{a_k} G^{* b_k} \rr).    
\end{equation}

These \textit{mixed moments} have been studied by Starr and Walters \cite{StarrWalters}. They appear naturally in the moments of Girko's hermitized form $(G-z)(G-z)^*$, as well as in the quaternionic resolvent (see \cites{BurdaSwiech, NowTar2018}). They are known to be linked to relevant statistics of eigenvalues and eigenvectors of random matrices. We prove here that the multi-indexed sequence of mixed moments of any $*$-free word only depends on its length $m$.

\begin{theorem} \label{thm:mixed} Let $w$ be any $*$-free word of length $m$ and $(a_1, b_1, \dots, a_k, b_k)$ any $2k$-tuple of integers. The following limit holds: 
$$
\lim_{N \rightarrow \infty} \frac{1}{N} \E \Tr \lll( G_w^{a_1} G_w^{* b_1} G_w^{a_2} G_w^{* b_1} \cdots G_w^{a_k} G_w^{* b_k} \rr) = \mathscr{M}_{\Gin} (m a_1, m b_1, \dots, m a_k, m b_k ).
$$
\end{theorem}

\begin{proof}
It is clear from Theorem \ref{abcthm} and Definition (\ref{def_mixed_moments}) that
$$
\mathscr{M}_{\Gin} (m a_1, m b_1, \dots, m a_k, m b_k ) = a_{w_1},
$$
where
$$
w_1 = G^{m a_1} G^{* mb_1} \cdots G^{m a_k} G^{* m b_k}.
$$
One can represent the face $f_{w_1}$ in a way similar to Figure \ref{ex3FussCatalan}, with the only difference that the groups of $m$ edges do not alternate as regularly between star and non-star, but follow the pattern imposed by the integers $a_i, b_i$. The relevant parameter $ a_{w_1}$ is the number of spherical pairings, which, as was explained in Subsection \ref{spherebuilding}, is equal to non-crossing admissible pairings of the edges of $f_{w_1}$. The key observation here is that a star edge $e^*$ and a star-free edge $e$ can be paired if and only if
\begin{enumerate}[(i)]
    \item If $e^*$ is in position $j$ in its group of $m$ star edges, then $e$ is in position $m+1-j$ in its group, and
    \item In each of the sides of the line connecting $e^*$ and $e$, the number of groups of star edges is equal to the number of groups of star free edges.
\end{enumerate}
Indeed, in each side of the line between $e$ and $e^*$, the numbers of star and star-free edges need to be equal, hence they need to be equal modulo $m$.

\medskip


Let us now turn to a general $*$-free word $w = G_{i_1} \cdots G_{i_m}$ of length $m$. For the same reason as above,
$$
\lim_{N \rightarrow \infty} \frac{1}{N} \E \Tr \lll( G_w^{a_1} G_w^{* b_1} G_w^{a_2} G_w^{* b_1} \cdots G_w^{a_k} G_w^{* b_k} \rr)
=
a_{w_2},
$$
where
$$
w_2 = (G_{i_1} \cdots G_{i_m})^{a_1} (G_{i_1} \cdots G_{i_m})^{* b_1} \cdots (G_{i_1} \cdots G_{i_m})^{a_k} (G_{i_1} \cdots G_{i_m})^{*b_k}.
$$
As in the proof of Theorem \ref{thm:mainFC}, the additional constraint that every edge must be paired with an edge of the same index is redundant with the constraints (i) and (ii). Therefore, $ a_{w_1} = a_{w_2}$ as claimed.
\end{proof}

One can immediately derive from this general property the following facts.

\begin{fact}
For every $*$-free word $w$ and $z \in \C$, the first order limit of the moments of the hermitization matrix $H_w(z)=(G_w-z)(G_w-z)^*$ depends only on the length of $w$.
\end{fact}

\begin{fact}
For every $*$-free word $w$, the first order limit of the moments of the Hermitian matrix $G_w + G_w^*$ only depend on the length of $w$.
\end{fact}

The link with hermitization as well as the quaternionic resolvant method (see \cite{BurdaSwiech}) suggests that the length of the $*$-free word is a key parameter for relevant statistics of eigenvalues and eigenvectors. We also note that convergence of the empirical measure of eigenvalues to the $m$-twisted circular law is known for powers as well as products (see \cite{BurdaNowak}). Another fact of great relevance is that for finite $N$ and a quite general model of bi-unitarily invariant ensembles, the distribution of eigenvalues determines the distribution of singular values, and vice versa \cite{Kieburg2016}. This has been specifically applied to products of i.i.d.\! matrices in \cite{Kieburg2019}. These and other known results suggest that the length of the $*$-free word is the only relevant parameter for eigenvalues at first order. Such a claim, however, would require a more intricate proof and involve objects beyond the scope of this work. 

\subsection{Fluctuation of eigenvalues of $*$-free words}

We call coperiod of a word $w$ the largest integer $k$ for which $w$ is a $k$-th power.
$$
\cop(w) = \max \{ k \ | \ \exists w_0, \ w=w_0^k \}
$$

Although the distribution of singular values at first order and more generally all mixed moments only depend on the length, it is the coperiod that appears to be the key parameter for fluctuations as well as repulsion of eigenvalues.

\begin{proposition}\label{StarfreeCLT}
For any $*$-free word $w$, 
$$
\Tr(G_w) \distconv \mathscr{N}_{\C}(0, \cop(w) )
$$
\end{proposition}

\begin{proof}
If $w$ is a $*$-free word, there is no admissible pairing of $w$ onto itself or a copy of itself, hence $a_w=c_w=0$. On the other hand, there are $\cop(w)$ spherical pairings of $w$ onto $w^*$ because the standard pairing of $w$ and $w^*$ that pairs each edge of $w$ with its mirror image in $w^*$ has a rotational symmetry of order $\cop(w)$ (see Figure \ref{example_cop}). Other admissible pairings are not spherical by the conditions that are explained in Subsection \ref{spherebuilding}. Therefore $b_w=\cop(w)$, and the result follows directly from Theorem \ref{abcthm}.
\end{proof}

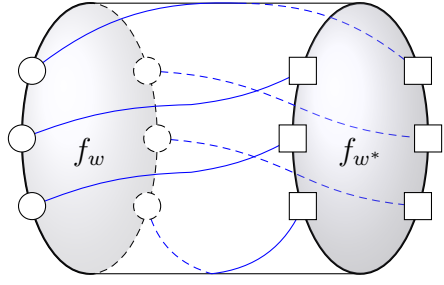
\begin{figure}[ht]
\begin{center}
\begin{tikzpicture}[scale=.9]
 \draw[thick] (-2,0) ellipse (1cm and 2cm);
 \draw[thick] (2,0) ellipse (1cm and 2cm); 
 \filldraw[fill=white,draw=white] (-1.9,0) ellipse (.99cm and 1.99cm);
 \draw[densely dashed] (-2,0) ellipse (1cm and 2cm);
    \shade[ball color=blue!10!white,opacity=0.20] (-2,0) ellipse (1cm and 2cm);
 \draw (-2,-.2) node {$f_w$} ;
 \draw (-2,2) -- (2,2);
 \draw (-2,-2) -- (2,-2);
   \draw[color=blue] (-3,0) to[bend left=10] (-.5,.5) to[bend right=10] (1.15,1); 
   \draw[color=blue] (-2.85,-1) to[bend left=10] (-.5,-.5) to[bend right=10] (1,0);
\draw[color=blue] (-2.85,1) to[bend left=20] (.2,2);
\draw[densely dashed,color=blue] (.2,2) to[bend left=20] (2.85,1);
\draw[densely dashed,color=blue] (-1.15,-1) to[bend right=30] (-0.2,-2);
\draw[color=blue] (-0.2,-2) to[bend right=30] (1.15,-1);
   \draw[densely dashed,color=blue] (-1,0) to[bend left=10] (1,-.5) to[bend right=10] (2.85,-1);
  \draw[densely dashed,color=blue] (-1.15,1) to[bend left=10] (1,.5) to[bend right=10] (3,0);
 \shade[ball color=blue!10!white,opacity=0.30] (2,0) ellipse (1cm and 2cm);
 \draw (2,-.2) node {$f_{w^*}$} ;
    \filldraw[densely dashed,fill=white,draw=black] (-1,0) circle (.2cm); 
    \filldraw[densely dashed,fill=white,draw=black] (-1.15,1) circle (.2cm); 
    \filldraw[densely dashed,fill=white,draw=black] (-1.15,-1) circle (.2cm); 
    \filldraw[fill=white,draw=black] (-3,0) circle (.2cm);
    \filldraw[fill=white,draw=black] (-2.85,1) circle (.2cm); 
    \filldraw[fill=white,draw=black] (-2.85,-1) circle (.2cm);
    \filldraw[fill=white,draw=black] (1-.2,0-.2) rectangle (1+.2,0+.2);
    \filldraw[fill=white,draw=black] (1.15-.2,1-.2) rectangle (1.15+.2,1+.2);    
    \filldraw[fill=white,draw=black] (1.15-.2,-1-.2) rectangle (1.15+.2,-1+.2);
    \filldraw[fill=white,draw=black] (3-.2,0-.2) rectangle (3+.2,0+.2);
    \filldraw[fill=white,draw=black] (2.85-.2,1-.2) rectangle (2.85+.2,1+.2);
    \filldraw[fill=white,draw=black] (2.85-.2,-1-.2) rectangle (2.85+.2,-1+.2);
\end{tikzpicture}
\caption{The number of admissible pairings of $w$ and $w^*$ that yield a sphere is $\cop(w)$ when $w$ is star-free.}
\label{example_cop}
\end{center}
\end{figure}

The proof of Theorem \ref{StarfreeCLT_joint} announced in the introduction can be directly inferred.

\begin{proof}[Proof of Theorem \ref{StarfreeCLT_joint}]
This follows from Propositions \ref{StarfreeCLT} and \ref{independence}, as powers of $w$ are star-free and trace-distinct. Then, 
$$ \lll( \Tr (G_w), \Tr (G_w^2), \dots, \Tr (G_w^n) \rr) \distconv (Z_1, \dots, Z_n)$$
where the $Z_j$'s are independent centered Gaussian variables with variances $\cop(w^j) = j \cop (w)$.
\end{proof}

We derive from this a result for polynomial linear statistics.

\begin{proposition}
Let $w$ be a $*$-free word, $f$ a polynomial, and $S_N(f) = \sum_{i=1}^N f(\la_i)$ the linear statistics of $G_w$ with respect to $f$. Then,
$$
S_N(f) - N f(0) \xrightarrow[N \rightarrow \infty]{d} \mathscr{N}_{\C} (0, \sigma_f^2)
$$
where 
$$
\sigma_f^2 = \cop(w) \int_{\D} |\nabla f(z)|^2 \dd m(z).
$$
\end{proposition}

\begin{proof} It follows from Theorem \ref{StarfreeCLT_joint} that, if $f(z)=\sum_{k=1}^n a_k z^k$, the centered linear statistics 
$$
S_N(f) - N f(0) = \sum_{k=1}^n a_k \Tr(G_w^k)
$$
converge to a Gaussian variable with variance
$$
\sigma_f^2 = \sum_{k=1}^n  k a_k^2 \cop (w) =  \cop(w) \int_{\D} |\nabla f(z)|^2 \dd m(z),
$$
which is the claim.
\end{proof}

This expression of the variance is reminiscent of Rider and Virag's results \cite{RiderVirag} for the fluctuations of linear statistics of a single complex Ginibre matrix. For general $*$-free words, we are only able to analyse polynomial test functions of the fluctuations. \medskip

It is known that eigenvalues of products of i.i.d.\! matrices form determinantal point processes (see \cites{AkemannIpsen,Ipsenthesis}). In particular, powers of the eigenvalues of a product of independent Ginibre matrices can be decomposed in independent blocks (see \cite{DubachPowers}), but no such property is known to hold for general $*$-free words.
\section{Extension to general non-Hermitian matrices}
\label{extension_section}

In this Section we extend our results to words of general non-Hermitian matrices with i.i.d.\! entries, sparse matrices and band matrices. 
\begin{enumerate}
    \item We extend all our results to matrices with i.i.d.\! entries under a second-moment and fourth-moment matching condition.
    (Section \ref{fourthmoment_section}).
     \item We formulate a genus expansion formula which holds for band matrices with i.i.d.\! Gaussian entries. Consequently, the results from Section \ref{SurfacesSection} extend to the band case.
    (Section \ref{band_section})
        \item We establish a weaker version of the first-order limits to sparse matrices with i.i.d.\! entries and optimal sparsity parameter, under a second-moment matching condition. 
    (Section \ref{sparse_section})
\end{enumerate}

\subsection{Words of non-Hermitian matrices}\label{fourthmoment_section}
In this subsection, we extend the results of the previous sections to $*$-words of random matrices with i.i.d.\! entries under moment matching conditions. Let $Z$ be a complex random variable whose distribution satisfies the following conditions:
\begin{enumerate}
    \item $\E|Z|^2=1.$
    \item $\E|Z|^4=2.$
    \item $\E Z^i \overline{Z}^j =0,~~\forall i \neq j.$
    \item $\E |Z|^{2j}< \infty,~~\forall j.$
\end{enumerate}
Let $M_1,M_2,...$ be random non-Hermitian matrices of order $N$ with i.i.d.\! $Z/\sqrt{N}$ distributed entries. Note that if $Z$ is a complex Gaussian distribution $M_i$ is a complex Ginibre matrix. Given a $*$-word $w$, we consider the random matrix $M_w$ by naturally replacing every letter $G_i$ (resp. $G_i^*$) with the random matrix $M_i$ (resp. $M_i^*$ the conjugate transpose of $M_i$). The main result of this subsection an approximate genus expansion formula for these random matrices.

We use the terminology and notations that were introduced in Section \ref{SurfacesSection} and in the proof of Theorem \ref{genusexpansion}.

\begin{definition}
 Let $w_1,...,w_k$ be $*$-words, $\phi\in\AdPair(w_1,...,w_k)$ an admissible pairing and $I$ a balanced indexing that is compatible with $\phi$. 
\begin{itemize}
        \item $I$ is called a {\em generic} indexing of $\phi$ if it is supported on $V(S_{\phi})$ indices from $\{1,...,N\}$. In other words, $I$ induces a distinct label for every vertex of $S_{\phi}$.
    \item $\phi$ is called {\em non-degenerate} if $V(S_{\phi})\ge V(S_{\phi'})$ for every $\phi'\in\AdPair(w_1,...,w_k)$.
\end{itemize}
\end{definition}
 
\begin{lemma}\label{lem:bad_phi}
Let $w_1,...,w_k$ be $*$-words, $\phi\in\AdPair(w_1,...,w_k)$ and $I$ a generic indexing of $\phi$. If there is a variable $v\in\mathscr{V}(I)$ with multiplicity $m_v\ge 3$ then $\phi$ is degenerate.
\end{lemma}

\begin{proof}[Proof of Lemma \ref{lem:bad_phi}]
The surface $S_{\phi}$ is created by gluing the polygonal faces $f_{w_1},...,f_{w_k}$ according to $\phi$. Recall that prior to the gluing each of $f_{w_1},...,f_{w_k}$ is oriented, which induces a cyclic ordering on the edges of each face. A {\em star} of $\phi$ is a cyclic sequence $T=(e_1,e_2,...,e_{2p-1},e_{2p})$ where each element is an edge of one of the polygonal faces $f_{w_1},...,f_{w_k}$, such that for every $1\le i\le p$,
\begin{enumerate}
    \item $e_{2i-1}$ and $e_{2i}$ belong to the same face $f_{w_j}$ and $e_{2i}$ is the immediate  successor of $e_{2i-1}$ in the orientation of $f_{w_j}$.
    \item $e_{2i}$ and $e_{2i+1}$ are paired by $\phi$.
\end{enumerate}
Note that every vertex $x$ of $S_{\phi}$ corresponds to a star $T(x)$ of $\phi$. Our strategy is to construct an admissible pairing $\phi'$ that has more stars than $\phi$. 

We label every edge in $S_{\phi}$ with an index $r$ according to the labels $G_r,G_r^*$ of its corresponding edges of the polygonal faces. The assumption that there is a variable $v=(r,s,t)\in\mathscr{V}(I)$ such that $m_v\ge 3$ where $I$ is generic implies that there are two vertices $x,y\in S_{\phi}$ that have $3$ edges of the same index $r$ between them. Therefore, there are $3$ distinct indices $1\le i_1<i_2<i_3\le p$ in the star $T(x)=(e_1,...,e_{2p})$ of $x$,  such that the labels of all the edges $e_{2i_j},e_{2i_j+1},j=1,2,3,$ have index $r$. Since there are $3$ such indices, we can suppose, without loss of generality, that $e_{2i_1}$ and $e_{2i_2}$ are labeled $G_r$, whereas $e_{2i_1+1}$ and $e_{2i_2+1}$ are labeled $G_r^*$. Note that the assumption $m_v\ge 3$ is essential here.

The admissible pairing $\phi'$ is constructed by the following alteration of $\phi$. Namely, we pair $e_{2i_1}$ with $e_{2i_2+1}$ and $e_{2i_2}$ and $e_{2i_1+1}$. Note that $\phi'$ is indeed an admissible pairing because of our assumption on the labels of the altered edges. In addition, the only stars that are affected by this alteration are $T(x)$ and $T(y)$, and each of them is split into two stars. For instance, $T(x)$ is split into $(e_1,..,e_{2i_1},e_{2i_2+1},...,e_{2p})$ and $(e_{2i_1+1},...,e_{2i_2})$. In conclusion, $V(S_{\phi'})>V(S_{\phi})$, and $\phi$ is degenerate.
\end{proof}

We now state the main result of this subsection.
\begin{theorem}\label{thm:fourth_moment_extension}
Suppose that $Z$ is a complex random variable satisfying assumptions (i)-(iv), $w_1,...,w_k$ are $*$-words, and $M_{w_1},...,M_{w_k}$ the corresponding $*$-words of $N\times N$ random matrices with $\frac 1{\sqrt{N}}Z$-distributed i.i.d.\! entries. Then,
$$
\E \left( \prod_{i=1}^k \Tr M_{w_i} \right) 
=   \left( 1+ \OO \left( \frac{1}{N} \right) \right)\sum_{\phi \in \AdPair(w_1, \dots, w_k)} {N^{2c({S_\phi}) -k - 2 g (S_\phi)}}.
$$
\end{theorem}

\begin{proof}
Let $\tilde a:=\tilde a_{w_1,...,w_k}$ denote the number of non-degenerate admissible pairings of $w_1,...,w_k$ and let $\tilde V$ denote the number of vertices in the surface induced by such a pairing. We need to show that $$E \left( \prod_{i=1}^k \Tr M_{w_i} \right) 
=   \left( 1+ \OO \left( \frac{1}{N} \right) \right)\tilde aN^{\tilde V-m/2}.$$
We use the notations in the proof of Theorem \ref{genusexpansion}. Denote
\[M(I):= \prod_{j=1}^{k}\prod_{l=1}^{m_j} M_{i_{j,l},i_{j,l+1}}^{(j,l)}\]
where $I\in\mathscr{I}$ is an indexation. By the normalization of the random matrices $M_{w_i}$, there holds 
$\E[M(I)]=N^{-\frac m2}\prod_{v\in\mathscr{V}(I)}\E[|Z|^{2m_v}]$ if $I$ is balanced, and $\E[M(I)]=0$ otherwise. Consider the relation $I\rightsquigarrow \phi$ which asserts that $I$ is a balanced indexing that is compatible with an admissible pairing $\phi$. We expand
\[
\E \left( \prod_{i=1}^k \Tr M_{w_i} \right) = 
\sum_{\phi\in\AdPair(w_1,...,w_k)}\sum_{I:I\rightsquigarrow \phi} \frac{\E[M(I)]}{\prod_{v\in\mathscr{V}(I)}m_v!}
=N^{-\frac m2} \sum_{\phi}\sum_{I} \prod_v\frac{\E[|Z|^{2m_v}]}{m_v!}~, 
\]
where the first equality follows from the fact that each balanced indexing $I$ is compatible with $\prod_{v\in\mathscr{V}(I)}m_v!$ admissible pairings. 

Fix a pairing $\phi$ and consider the sum \[\sum_{I\rightsquigarrow\phi} \prod_{v\in\mathscr{V}(I)}\frac{\E[|Z|^{2m_v}]}{m_v!}.\] If $\phi$ is degenerate then there are only $\OO(N^{\tilde V-1})$ indexations compatible with $\phi$. Otherwise, if $\phi$ is non-degenerate, we split the summation to generic and non-generic indexations. By definition, there are only $\OO(N^{\tilde V-1})$ non-generic indexations. Moreover, $m_v\le 2$ for every $v\in\mathscr{V}(I)$ provided that $I$ is generic by Lemma \ref{lem:bad_phi}. In consequence, by our moment assumptions on $Z$, each of the $(1+\OO(1/N))N^{\tilde V}$ generic indexations of $\phi$ contribute exactly $1$ to the sum. The claim follows since there are exactly $\tilde a$ such non-degenerate pairings.
\end{proof}

By reiterating the arguments in Section \ref{Applications}, Theorem \ref{thm:fourth_moment_extension} implies that all the results in that section (see Figure \ref{fig:plan_of_paper}) also hold for $*$-words of general non-Hemitian matrices  with i.i.d.\! entries assuming the distribution of the entries satisfies assumptions (i)-(iv). In particular, the theorem implies the universality of the Fuss-Catalan limit for the singular values that is described in Theorem \ref{thm:mainFC} and the dependency of the mixed matrix moments of $M_w$ in its length that is mentioned in Theorem \ref{thm:mixed}. There are many similar results in random matrix theory. Notably, the universality of the singular value distribution as well as eigenvalue distribution was established in \cite{Gotze2015} for general matrix-valued functions of independent matrices, under more involved assumptions.

\subsection{Genus expansion for band matrices}\label{band_section}

We present in this section an extension of Theorem \ref{abcthm} to $*$-words of Gaussian band matrices, motivated by the recent work of Au ~\cite{Au2019} on infinitesimal freeness. We consider a band parameter $(b_N)$ such that 
\be\label{band_condition_1}
\frac{b_N}{{N}^{1/3}} \rightarrow \infty, \qquad
\frac{b_N}{N} \rightarrow \la \geq 0.
\ee
The assumption $b_N\gg N^{1/3}$ is required so that $a_w N$, by which we centralize, is an accurate enough approximation of $\E \Tr G_w$.
The band condition we consider is periodic. Namely, we let $d_N(i,j)=\min (|i-j|,N-|i-j|)$ for every $1\le i,j \le N$, and define 
an $N\times N$ Gaussian band matrix $G$ with band parameter $b_N$  by letting
\begin{equation}\label{band_condition_2}
    G_{i,j} := \frac{1}{\sqrt{l_N}} \mathds{1}_{d_N(i,j) \leq b_N} Z_{i,j}
\end{equation}
where $
    l_N := \min( 2b_N+1, N)
$ is the number of indices within cyclic distance $b_N$ of any given index and $Z_{i,j}$ are i.i.d.\! standard complex Gaussian variables.

 \medskip

Let $G_r, r\ge 1$, be Gaussian band matrices distributed according to (\ref{band_condition_2}), $w_1, \dots w_k$ be $*$-words on the formal alphabet defined in the introduction, and $G_{w_i},~i=1,...,k$ the corresponding $*$-words of complex Gaussian band matrices. For every admissible pairing $\phi\in\AdPair(w_1,...,w_k)$, we denote by $\mathscr{I}_N (\phi)\subset \mathscr{I}$ the subset of indexations $I$  that are compatible with $\phi$ such that for every adjacent vertices $x,y$ in $S_{\phi}$ the labels $1\le i\le N$ and $1 \le j \le N$ that $I$ induces on $x$ and $y$ respectively satisfy $d_N(i,j) \leq b_N$.  Theorem \ref{genusexpansion} and its proof extend to the band case directly. A band version of genus expansion first appeared in \cite{Au2019}, which inspired the extension of our results to the band case. 

\begin{proposition}[Band Genus expansion]\label{prop:band_genus_expansion}
Let $w_1,...,w_k$ be $*$-words of total length $m$ and $G_{w_1},....,G_{w_k}$ the corresponding $*$-words of $N\times N$ complex Gaussian band matrices with band parameter $b_N$. Then,
\begin{equation}\label{band_genus_expansion}
\E \left( \prod_{i=1}^k \tr G_{w_i} \right)
=
\sum_{\phi \in \AdPair(w_1, \dots, w_k)} |\mathscr{I}_N (\phi)| l_N^{-m/2}
\end{equation}
\end{proposition}

The first step in applying this formula is to estimate the asymptotic growth of $|\mathscr{I}_N(\phi)|$. It is clear that $|\mathscr{I}_N(\phi)|=N^{V(S_{\phi})}$ if $l_N=N$. However, if $l_N<N$, $|\mathscr{I}(\phi)|$ depends on the structure of the graph $\Ga_{\phi}$ that results by pairing the boundaries of the polygonal faces $f_{w_1},...,f_{w_k}$ according to $\phi$. For every connected $V$-vertex graph $\Gamma$ and $0\le \la\le 1/2$ we associate a $V-1$ dimensional body $D_{\Gamma}^{(\la)}$ in $\R^V$ as follows. 
\begin{enumerate}
\item If $0<\la\le 1/2$, we let $I_{\la}:= \left[-\frac{1}{4\la},\frac{1}{4\la}  \right]$ and define:
$$
D_{\Gamma}^{(\la)} :=
\left\{ 
(x_1, \dots, x_V) \in I_{\la}^V \ | \ x_1=0, \ \forall (i,j) \in E, \  d(x_i, x_j) < 1/2
\right\}
$$
where $d(x,y) = \min (|x-y|,|I_\la|-|x-y|)$ is the cyclic distance on $I_{\la}$ and $E$ denotes the edge set of $\Gamma$.
\item If $\la =0$,
$$
D_{\Gamma}^{(0)} :=
\left\{
(x_1, \dots, x_V) \in \R^V \ | \ x_1=0, \ \forall (i,j) \in E, \  |x_i - x_j|< 1/2 
\right\}.
$$

\end{enumerate}
For a general graph $\Gamma$, we define
$$
D_{\Gamma}^{(\la)} : = D_{\Gamma_1}^{(\la)} \times \cdots \times D_{\Gamma_k}^{(\la)},
$$
where $\Gamma_i$ are the connected components of $\Gamma$. 
\begin{lemma}\label{volume_lemma} 
Let $w_1,...,w_k$ be $*$-words, $\phi \in \AdPair(w_1,...,w_k)$ and $b_N$ a band parameter satisfying $b_N/N\to \la\in[0,1/2]$ as $N\to\infty$. Then,
$$
|\mathscr{I}_N(\phi)| = \alpha_{\phi} N^{c} l_N^{V-c} \left(1+ O\left(\frac{1}{l_N} \right) \right)
$$
where $\alpha_{\phi}$ is $(V-1)$-dimensional volume of $D_{\Gamma_{\phi}^({\la})}$ and $l_N=2b_N+1$.
\end{lemma}

\begin{proof}
Clearly, it suffices to prove the statement for the case that  $\Gamma_{\phi}$ is connected. Since the band condition is translation invariant we have
\begin{align*}
|\mathscr{I}_N (\phi)| & = N \times |\{ (I_1,\dots, I_V) \in \{1,...,N\}^V \ : \ I_1= 0 ,~\forall (i,j)\in E \ d_N(I_i, I_j) < b_N \} |
\end{align*}
where $d_N$ denotes the cyclic distance on $\{1,...,N\}.$ This number is equal to the number of integral points that fall in $l_N\cdot D_{\Gamma}^{(b_N/N)}$, and the claim follows by evaluating the corresponding Ehrhart's polynomial at $l_N$.
\end{proof}

\begin{remark}\label{atom_rmk} If $\phi$ is a spherical admissible pairing of the edges of a single $*$-word $w$ then  $\Ga_{\phi}$ is a tree. A straightforward counting argument shows that in such case $|\mathscr{I}_N(\phi)|= N l_N^{V-1}$, which is consistent with the fact that $\alpha_{\phi}=1$.
\end{remark}

By combining Lemma \ref{volume_lemma} and Proposition \ref{prop:band_genus_expansion}, the band genus expansion takes the following useful form:
\begin{equation}\label{band_genus_expansion_2}
\E \left( \prod_{i=1}^k \tr G_{w_i} \right)
=
\sum_{\phi \in \AdPair(w_1, \dots, w_k)}
    \al_{\phi}
N^c l_N^{c-2g-k}
\left(1 + \OO(1/l_N) \right).
\end{equation}

Note that if $l_N=N$, then each term in the sum is $N^{2c - 2g - k}$
consistently with the usual genus expansion. 

\begin{theorem}[Central Limit Theorem for Gaussian band matrices]\label{band_CLT}
Let $w$ be a $*$-word and $G_w$ the corresponding $*$-word of $N\times N$ complex Gaussian band matrices with band parameter $b_N$ satisfying  $N^{-1/3} b_N \rightarrow \infty$. Then,
\begin{equation}\label{band_first_order}
\E \tr G_{w} = a_{w} N + \OO \left( \frac{N}{l_N^2} \right),
\end{equation}
and
\begin{equation}
\label{band_second_order}
 \sqrt{\frac{l_N}{N}} \left( \tr G_{w} - a_{w} N \right) \distconv X + i Y
\end{equation}
where $X,Y$ are two independent real centered Gaussian variables with variances $\frac{b'_w+c'_w}{2}$ and $\frac{b'_w-c'_w}{2}$ respectively. The values of $b'_w, c'_w$ are given by the sums of $\al_{\phi}$ over the spherical pairings $\phi$ in $\AdPair(w,w^*),\AdPair(w,w)$ respectively.
\end{theorem}

\begin{proof}
The proof goes along the same lines as the proof of Theorem \ref{abcthm}, so we only emphasize the differences with the full-band case. 
To derive Equation (\ref{band_first_order}), we use Remark \ref{atom_rmk} which states that each of the $a_w$ spherical pairings of $w$ contribute $N$ when expanding $\E[\Tr G_w]$  according to (\ref{band_genus_expansion}). By expansion (\ref{band_genus_expansion_2}) we see that all other pairings contribute at most $O(N/l_N^2)$.

Lemmas \ref{lemma:centralized} and \ref{lemma:diatomic} can be extended to $*$-words of Gaussian band matrices. Namely, let $w_1,...,w_k$ be $*$-words and $G_{w_i},~i=1,...,k$, the corresponding $*$-words of Gaussian band matrices with band parameter $b_N$. Then,
\begin{equation}
\E\left[\prod_{i=1}^{k}(\Tr G_{w_i}-a_{w_i}N)\right]=\sum_{\phi\in\AdPair_{AF}(w_1,...,w_k)}\alpha_{\phi}N^cl_N^{c-2g-k}(1+\OO(1/l_N)).
\label{band_centralized_1}
\end{equation}
The proof follows by an inclusion-exclusion argument similar to Lemma \ref{lemma:centralized}. Recall that $\phi\in\AdPair(w_1,...,w_k)$ is bi-atomic  if $S_{\phi}$ consists of $k/2$ spheres. From equation (\ref{band_centralized_1}) we derive
\begin{equation}
\E\left[\prod_{i=1}^{k}(\Tr G_{w_i}-a_{w_i}N)\right]=\sum_{\mbox{$\phi$ bi-atomic}}\alpha_{\phi}\left(\frac{N}{l_N}\right)^{k/2}\left(1+\OO\left(\frac{1}{l_N}+\frac N{l_N^3}\right)\right).\label{band_centralized_2}
\end{equation}
To derive (\ref{band_centralized_2}), recall inequality (\ref{eqn:gc2k}) which states that every atom-free pairing $\phi$ satisfies $g\ge 2c-k$. Therefore, the order of the contribution of a pairing $\phi$ with $c$ components to (\ref{band_centralized_1}) is at most $N^cl_N^{-3c+k} = (N/l_N)^{k/2}(N/l_N^3)^{c-k/2}.$ In particular, the sum is dominated by the bi-atomic pairings in which $c=k/2$ and $g=0.$

We now derive (\ref{band_second_order}) by the moments method. Namely, we have that for every $k,l$,
\[
\E[(l_N/N)^{(k+l)/2)}(\Tr G_w - a_wN)^k(\Tr G_w^* - a_wN)^l]
= \sum_{\phi}\alpha_{\phi} + o(1),
\]
where the summation is over bi-atomic pairings of $k$ copies of $w$ and $l$ copies of $w^*$. Here we use the assumption that $N^{-1/3}l_N\to \infty$. Similarly to Theorem \ref{abcthm}, the proof follows by noting that $\sum_{\phi}\alpha_{\phi}$ is equal to the sum, over all pairings of $k$ $w$-labelled faces and $l$ $w^*$ labelled faces, of ${b'_w}^{d}{c'_w}^{(k+l)/2-d}$, where $d$ is the number of $w$-labelled faces that are matched with a $w^*$-labelled face.
\end{proof}

Theorem \ref{band_CLT} implies that Theorem \ref{thm:mainFC} also applies to $*$-free words of band matrices with band parameter $b_N\gg N^{1/3}$. Indeed, as we show in subsection \ref{subsec:FC}, Theorem \ref{thm:mainFC} is implied by the almost-sure convergence of $N^{-1}\Tr(G_{(ww^*)^k})$ to $a_{(ww^*)^k}$ as $N\to\infty$ for every $*$-free word $w$ and $k\ge 1$. In the band case, Theorem \ref{band_CLT} implies that
$$
\mathrm{Var} \left( \frac{1}{N} \Tr\left( (G_w G_{w}^*)^k \right) - a_{(ww^*)^k} \right) = \OO((Nl_N)^{-1}),
$$
and the almost sure convergence follows, as $l_N \gg N^{1/3}$.

In addition, Theorem \ref{band_CLT} yields a generalization of a recent result by Jana \cite{Jana_band} to any $*$-free word in complex Gaussian band matrices. We state it below in the small band regime.

\begin{corollary}
For any $*$-free word $G_w = G_{i_1} \cdots G_{i_m}$ of band complex Gaussian matrices with bandwidth 
$N^{1/3}\ll b_N \ll N$, there holds
$$
\sqrt{\frac{l_N}{N}} \Tr G_w \xrightarrow[N \rightarrow \infty]{d} \mathscr{N}_{\C}(0, \mathrm{cop}(w) \al_{m} )
$$
where $m$ is the length of $w$, $\mathrm{cop}(w)$ its coperiod, and
$$
\al_m :=
\frac{1}{\pi} 
\int_{-\infty}^{\infty}
{\mathrm{sinc} ( \theta )^{m}}
{\dd \theta}.
$$
\end{corollary}

\begin{proof}
Convergence in distribution is given by Theorem \ref{band_CLT}; we are left to compute the specific values of $a_w, b_w', c_w'$ in the band case, when $w$ is star-free. As in the non-band case, $a_w = c_w' =0$ because there is no admissible pairing in the corresponding sets. For $b_w'$, there are $\cop(w)$ admissible pairings that are spherical, and for each of them, $\Ga_{\phi}$ is a cycle of size $m$. We now compute the value of $\al_{\phi}$ in that case, which can be reformulated as a classical exercise on circulant matrices. \medskip

For any $N,m$ and finite sequence $a_0, \dots, a_{N-1}$ of real positive numbers such that $\sum a_i=1$, we consider the circulant (Toeplitz) matrix $A$ such that $A_{i,j} = a_{j-i [N]}$. We also consider the random walk on $\{1,...,N\}$ driven by these probabilities, which can be seen as an increasing random walk on $\Z$, beginning at $0$ and taken modulo $N$. Notice that
$$
\frac{1}{N} \tr A^m = \frac{1}{N} \sum_{i_1, \dots, i_m} A_{i_1,i_2} \cdots A_{i_m, i_1}
 = \PP \left( S_{m+1} = 0 [N] \right).
$$
We introduce the generating function
$$
\phi(t) = \sum_{i=0}^{N-1} a_i t^i,
$$
so that the generating function of $S_m$ (on $\Z$) is $\phi(t)^m$, and so
$$
\tr A^m
= N \PP \left( S_{m} = 0 [N] \right) 
= N \sum_{k \geq 0} \left[ \phi(t)^{m} \right]_{kN}
= \frac{N}{2 \pi} \sum_{k \geq 0} \int_{-\pi}^{\pi} \phi(e^{i\theta})^{m} e^{- i kN \theta} \dd \theta
$$
which is an exact formula for the moments of the circulant matrix $A$. We now apply it to our problem by taking $a_i$ to match the band condition. \medskip

We consider i.i.d.\! variables $X_i$ uniform on $[-b_N,b_N]$ and the corresponding partial sums $S_m$. The matrix $A$ is defined as above, with $a_i = \de_{d_N(0,i)<b_N} $. Recall that we assume $b_N = \oo(N)$. We need to compute the term $k=0$, since $m$ is fixed and $N$ is diverging $[\phi(t)^m]_{kN}=0$ for a sufficiently large $N$. This argument yields
$$
|\mathscr{I}_N({\phi})|= \tr A^m = N \PP \left[ S_{m} = 0 [N] \right] = \frac{N}{2 \pi} \int_{-\pi}^{\pi} \phi(e^{i\theta})^{m} \dd \theta
$$
where 
$$
\phi(e^{i \theta}) = \sum_{k=-b_N}^{b_N} e^{i k \theta} 
=
\frac{\sin\left(\frac{l_N \theta}{2} \right)}{\sin\left( \frac{\theta}{ 2 } \right)}.
$$
This yields
$$
|\mathscr{I}_N({\phi})|= \frac{N}{2 \pi} \int_{-\pi}^{\pi} \frac{\sin^{m}\left(\frac{l_N \theta}{2} \right)}{\sin^{m}\left( \frac{\theta}{ 2 } \right)} \dd \theta
= \frac{N}{\pi l_N} \int_{-\infty}^{\infty} \frac{\sin^{m}\left({\theta} \right)}{\sin^{m}\left( \frac{\theta}{ l_N } \right)} \mathds{1}_{[-\pi l_N, \pi l_N]}(\theta) {\dd \theta}
$$
By dominated convergence using classical inequalities,
$$
\frac{1}{N l_N^{m-1} } |\mathscr{I}_N({\phi})|
\xrightarrow[N \rightarrow \infty]{}
\frac{1}{\pi} 
\int_{-\infty}^{\infty}
{\mathrm{sinc} ( \theta )^{m}}
{\dd \theta},
$$
which yields the claim.
\end{proof}

We refer the reader to \cite{Jana_band} for other results of this type (when $m=1$) as well as a combinatorial interpretation of the sequence $\al_m$.
\subsection{Extension to sparse random matrices with optimal sparsity }\label{sparse_section}
In this subsection, we consider $*$-words of random {\em sparse} matrices with i.i.d.\! entries. Let $Z$ be a complex random variable whose distribution satisfies conditions (i),(iii),(iv) that appear in the beginning of subsection \ref{fourthmoment_section}.
Let $0\le p=p_N\le 1$ so that $Np_N$ is diverging as $N \to \infty$. We consider random non-Hermitian sparse matrices $M$ of order $N$ with i.i.d.\! entries defined by
\begin{equation}\label{sparseentries}
M_{i,j} = \frac{1}{\sqrt{N p_N}} B_{i,j} Z_{i,j}
\end{equation}
where $Z_{i,j}$ are i.i.d.\! $Z$-distributed random variables and $B_{i,j}$ are independent Bernoulli variables with parameter $p_N$. The moments of $M_{i,j}$ are given by
$$
\E |M_{i,j}|^{2k} =  \frac{1}{N^k p_N^{k-1}} \E |Z|^{2k} 
$$
so that, in particular, the second moment matches the Gaussian case. We note that complex Ginibre is a special case, by taking $Z$ to be a complex Gaussian variable and $p_N=1$.

Given a $*$-word $w$, we consider the random matrix $M_w$ by naturally replacing every letter $G_i$ (resp. $G_i^*$) by the random matrix $M_i$ (resp. $M_i^*$ the conjugate transpose of $M_i$) where $M_1,M_2,...$ are i.i.d.\! matrices distributed according to (\ref{sparseentries}). The main result of this subsection is that under the assumptions on the distribution of $Z$ and the growth of $p_N$, $\E\Tr M_w$ is close to $\E\Tr G_w$ at first order.

\begin{proposition}\label{pro:MG}
Let $B$ be a Bernoulli random variable with parameter $p_N$ where $Np_N\to\infty$, and $Z$ a complex random variable satisfying assumptions (i),(iii),(iv). Suppose that $w$ is a $*$-word, and $M_w$ the corresponding $*$-word of $N\times N$ random matrices distributed according to (\ref{sparseentries}). Then,
$$
\E \left( \Tr M_{w} \right) = a_wN
+ \OO \left( \frac{1}{p_N} \right).
$$
\end{proposition}
The proposition implies that a weaker version of Theorem \ref{thm:mainFC} holds for $*$-words of random {\em sparse} matrices. Namely, the convergence \textit{in expectation only} of the empirical measure of the squared singular values to the Fuss-Catalan distribution. Theorem \ref{thm:mixed} concerning mixed moments can also be extended to the sparse case. The sparsity assumption $Np_N \rightarrow \infty$ is optimal and appears in other works related to sparse non-Hermitian matrices: for instance, convergence to the circular law for one sparse non-Hermitian matrix with i.i.d.\! entries has been established in \cite{Rudelson2}. This is coherent with our result on mixed moments of words, which suggests that the empirical eigenvalue distribution of any word of sparse matrices under the above sparsity assumption converges to the same limit as it does in the non-sparse model.

\begin{proof}
We again use the terminology that is used in the proof of Theorem \ref{genusexpansion}. Let $m$ be the length of $w$ and let $\mathscr{I} \subseteq \{1,...,N\}^m$ be the set of indexations. For every indexation $I=(i_1,...,i_m)\in\mathscr{I}$, we let 
\[
M(I):=M^{(1)}_{i_1,i_2}\cdots M^{(m)}_{i_m,i_1}~~\mbox{  and  
 }~~ G(I):=G^{(1)}_{i_1,i_2}\cdots G^{(m)}_{i_m,i_1},\] where $M^{(i)}$ and $G^{(i)}$ are the $i$-th matrices in $M_w$ and $G_w$ respectively. Note that each $M(I)$ is a product of $m$ random variables in $\mathscr{V}:=\{(M_r)_{s,t}~:~t\ge 1,~1\le s,t\le N\}$, or their conjugates. An indexation $I$ is called {balanced} if every variable in $\mathscr{V}$ appear in $M(I)$ the same number of times as its conjugate. For every indexation $I$ in the subset $\mathscr{I}_B\subseteq\mathscr{I}$ of balanced indexations, we denote by $l(I)$ the number of distinct variables of $\mathscr{V}$ that appear in $M(I)$. Clearly, $1\le l(I) \le m/2$ for every $I\in\mathscr{I}_B$. \medskip
 
Note that $\E[M(I)]=\E[G(I)]=0$ if $I$ is not balanced since the unbalanced moments of $Z$ and of a complex Gaussian variable vanish. In addition, if $I$ is balanced and $l(I)=m/2$ then $\E[M(I)]=\E[G(I)]=N^{-m/2}$. Indeed, in such case $\E[M(I)]$ is a product of $m/2$ second moments of independent $M_{i,j}$-distributed variables that match the $m/2$ second moments of independent Gaussian variables which appear in $\E[G(I)]$. Moreover, if $I$ is balanced and $1\le l(I)<m/2$, let $2k_1+\cdots+2k_{l(I)}=m$ denote the multiplicities of the variables in $M(I)$. Therefore,
 \begin{equation}
\label{lbound}
 \big|\E[M(I)]-\E[G(I)]\big| =\left| \prod_{i=1}^{l(I)}\frac{\E[Z]^{2k_i}}{N^{k_i}p_N^{k_i-1}}-\prod_{i=1}^{l(I)}\frac{\E[G]^{2k_i}}{N^{k_i}} \right|= \OO \left( p_N^{l(I)-m/2}N^{-m/2} \right),
 \end{equation}
 where $G$ denotes a Gaussian random variable. Note that the constant in the $\OO$-notation depends on $m$, the finite moments of $Z$ and the complex Gaussian distribution (but importantly not on $N$). \medskip
  
 Finally, for every $1\le l <m/2$, we denote by $\mathscr{I}_l$ the subset of balanced indexations $I$ for which $l(I)=l$, and bound the size of $\mathscr{I}_l$.

\begin{lemma}\label{indexbound1}
For every $1\le l<m/2$ there holds
$$
\left| \mathscr{I}_l  \right| = \OO( N^{l+1})
$$
\end{lemma}

\begin{proof}
We start by showing that if an indexation $I=(i_1,...,i_{m})$ involves at least $l+2$ different indices from $1,...,N$, then $l(I) \geq l+1$. Indeed, consider a graph whose vertex set $\{i_1,...,i_m\}$ is the indices that appear in $I$ (without repetition), two vertices being connected by an edge if they appear consecutively in $I$ (with the convention that $i_{m+1}=i_1$). $G$ is a quotient graph of the $m$-cycle and is therefore connected. In addition,  $G$ has at most $l(I)$ edges and at least $l+2$ vertices, hence we find that $l(I)\ge (l+2)-1$. 

Therefore, for every $I\in\mathscr{I}_l$ there are at most $l+1$ degrees of freedom in choosing the indices, and each index is an integer between $1$ and $N$, whence the claim.
\end{proof}

The proof of the proposition follows from expanding the trace, using the above observations, Equation (\ref{lbound}) and Lemma \ref{indexbound1}.
\begin{align*}
    |\E(\Tr M_w)-\E(\Tr G_w)| &\le \sum_{I\in\mathscr{I}} \left| M(I)-G(I) \right|
    =\sum_{l=1}^{m/2-1}\sum_{I\in\mathscr{I}_l} \left| M(I)-G(I) \right| \\
    &= N\cdot \sum_{l=1}^{m/2-1}\OO \left((Np_N)^{l-m/2}\right)
    =N \cdot \OO\left(\frac 1{Np_N}\right),
\end{align*}
and the claim follows from Theorem \ref{abcthm}.
\end{proof}


\subsection{Extension to words involving letters from different RMT ensembles}\label{polyletters}

The results of this paper focus on $*$-words over the complex Ginibre ensemble. In this subsection, we show how to generalize the genus expansion to $*$-words involving letters from other Gaussian matrix ensembles, as well as their transposed or conjugate counterparts. The definition of an admissible pairing has to include different sets of conditions depending on the letters that are being paired. Moreover, the construction of the resulting surface also depends on the letters since some identifications of edges need not respect the orientation of their faces. As a result, an admissible pairing may lead to a non-orientable surface. \medskip
We consider the following possible letters :
\begin{itemize}
    \item $G_i$, $G_i^*$, $G_i^t$ and $\overline{G_i}$ where $(G_i)_{i \geq 1}$ are i.i.d.\!  complex Ginibre matrices, 
    \item $R_i$ and $R_i^t$ where $(R_i)_{i \geq 1}$ are i.i.d.\! real Ginibre matrices,
    \item $H_i, \overline{H_i}$ where $(H_i)_{i \geq 1}$ i.i.d.\! GUE matrices, and
    \item $(S_i)_{i \geq 1}$ i.i.d.\! GOE matrices; matrices from different ensembles being also independent.
\end{itemize}
We review the specific rules of admissibility and orientation of each letter in Subsection \ref{polyletterconstraints}. \medskip

Any compact non-orientable connected surface $S$ has a non-orientable genus $ g_{\text{no}}(S)$, such that $S$ is homeomorphic to the connected sum of  $ g_{\text{no}}(S)$ real projective planes, and its Euler characteristic is 
$$
\chi(S) = 2- g_{\text{no}}(S)
$$
(note the absence of a factor $2$). Using this fact, the proof of Theorem \ref{genusexpansion} is easily extended to this more general framework. Let $w_1, \dots, w_k$ be $*$-words over alphabet of the letters defined above and $N$ an integer. We have the formula
\begin{equation}\label{nosurfaces}
    \E \left( \prod_{i=1}^k \Tr( G_{w_i}) \right) = \sum_{\phi \in \AdPair (w_1, \dots w_k)} {N^{2c({S_\phi}) -k - 2 g_{\text{o}} (S_\phi) - g_{\text{no}} (S_\phi)} },
\end{equation}
where $g_{\text{o}}$ is the sum of the genera of orientable connected components of $S_{\phi}$, and $ g_{\text{no}}$ is the corresponding sum for the non-orientable components. \medskip

One direct consequence of this formula is that for any such general $*$-word $w$,
\begin{equation}
\E \lll( \Tr G_w \rr) = a_w N + p_w + \underset{N \rightarrow \infty}{O} \lll( \frac{1}{N} \rr)
\end{equation}
where $a_w$ is, as before, the number of elements of $\AdPair(w)$ yielding a sphere $\S$, and $p_w$ the number of elements of $\AdPair(w)$ yielding a projective plane $\Proj$. The central limit theorem (Theorem \ref{abcthm}) and the related lemmas also hold for the centered variable $\Tr(G_w)-a_wN-p_w$, and can be proved in the very same way (the only difference is that the notion of atom has to be extended as to include projective planes components made from a single face).

\subsubsection{Types of letters and specific rules}\label{polyletterconstraints} We treat below each case separately.  We define the \textit{type} of a letter as its ensemble together with its integer index. For instance, we will speak of letters and edges of type $G_i$, $R_i$, $H_i$ and $S_i$.

\begin{enumerate}
    \item  \textbf{More general words over complex Ginibre.}
Before turning to other ensembles, we consider words in four categories of letters: $G_i, G_i^*, G_i^t, \overline{G_i}$ for each index $i$. \medskip

Admissible pairings are defined as pairings between edges of same type (integer index of the letter) and opposite sign (the sign being the presence or absence of conjugation, so that letters $G,G^t$ are considered as having the same sign, while $G^*, \overline{G}$ take the opposite sign). In the analogy with surfaces, the edges are endowed with a direct orientation when corresponding to letters without transposition $G, \overline{G}$ and indirect if corresponding to $G^*, G^t$. A pairing between compatible letters is now understood as identifying the corresponding edges in either a direct or indirect way, as illustrated on Figure \ref{compatibilityCGE} below.

\begin{figure}[ht]
\begin{center}
\begin{tikzpicture}[scale=1.3]
\draw  (-.9,1.14) -- (.9,1.14);
\draw  (-.9,1) -- (.9,1);
\draw  (-.9,-1.14) -- (.9,-1.14);
\draw  (-.9,-1) -- (.9,-1);
 \draw (-1,.9) -- (-1,.3) to[bend left=10] (-1.07,0) to[bend right=10] (-1.14,-.3) -- (-1.14,-.9);
 \draw (-1.14,.9) -- (-1.14,.3) to[bend right=10] (-1.07,0) to[bend left=10] (-1,-.3)-- (-1,-.9);
 \draw (1,.9) -- (1,.3) to[bend right=10] (1.07,0) to[bend left=10] (1.14,-.3) -- (1.14,-.9);
 \draw (1.14,.9) -- (1.14,.3) to[bend left=10] (1.07,0) to[bend right=10] (1,-.3)-- (1,-.9);
\draw (0,1.3) node {\small direct};
\draw (0,-.8) node {\small direct};
\draw (-1.35,0) node {
  \begin{tikzpicture}
      \node[rotate=90] {\small indirect};    
    \end{tikzpicture}
};
\draw (.8,0) node {
  \begin{tikzpicture}
 \node[rotate=90] {\small indirect};    
 \end{tikzpicture}
};
\draw[fill=white] (-1-.3,1-.3) rectangle (-1+.3,1+.3);
\draw[fill=white] (-1-.3,-1-.3) rectangle (-1+.3,-1+.3);
\draw[fill=white] (1-.3,1-.3) rectangle (1+.3,1+.3);
\draw[fill=white] (1-.3,-1-.3) rectangle (1+.3,-1+.3);
\draw (-1,1) node {$G_i$};
\draw (1,1) node {$G_i^*$};
\draw (1,-1) node {$G_i^t$};
\draw (-1,-1) node {$\overline{G_i}$};
\end{tikzpicture}
\end{center}
\caption{Compatibility between the four possible edges of a given type $G_i$.}
\label{compatibilityCGE}
\end{figure}
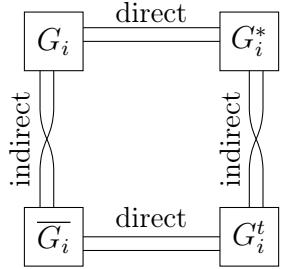

\item \textbf{Real Ginibre Ensemble.} The same technique can be applied to matrices with i.i.d.\! real Gaussian coefficients. In that case, any pairing that matches letters of the same type is admissible, and the pairs $R_i, R_i$ and $R_i^t,R_i^t$ are matched indirectly whereas $R_i$ is matched directly with $R_i^t$, as illustrated on Figure \ref{compatibilityRGE}. Genus expansion for real Ginibre matrices has been used by Redelmeier in \cite{Redelmeier} in order to establish real second order freeness.

\begin{figure}[ht]
\begin{center}
\begin{tikzpicture}[scale=1.2]
\draw  (-.9,1.1) -- (.9,1.1);
\draw  (-.9,.9) -- (.9,.9);
 \draw (-1.45,1.25) circle (.3cm);
  \draw (-1.45,.75) circle (.3cm);
  \draw (-1.45,1.25) circle (.5cm);
  \draw (-1.45,.75) circle (.5cm);
  \filldraw[fill=white, draw=white] (-2,.7) rectangle (-1,1.3);
 \draw (-1.945,1.31) to[bend right=20] (-1.85,1) to[bend left=20] (-1.745,.69);
 \draw (-1.745,1.31) to[bend left=20] (-1.85,1) to[bend right=20] (-1.945,.69);
 \draw (1.45,1.25) circle (.3cm);
  \draw (1.45,.75) circle (.3cm);
  \draw (1.45,1.25) circle (.5cm);
  \draw (1.45,.75) circle (.5cm);
  \filldraw[fill=white, draw=white] (2,.7) rectangle (1,1.3);
 \draw (1.945,1.31) to[bend left=20] (1.85,1) to[bend right=20] (1.745,.69);
 \draw (1.745,1.31) to[bend right=20] (1.85,1) to[bend left=20] (1.945,.69);
\draw (0,1.3) node {\small direct};
\draw (-2.2,1) node {
  \begin{tikzpicture}
      \node[rotate=90] {\small indirect};    
    \end{tikzpicture}
};
\draw (2.2,1) node {
  \begin{tikzpicture}
 \node[rotate=90] {\small indirect};    
 \end{tikzpicture}
};
\draw[fill=white] (-1-.3,1-.3) rectangle (-1+.3,1+.3);
\draw[fill=white] (1-.3,1-.3) rectangle (1+.3,1+.3);
\draw (-1,1) node {$R_i$};
\draw (1,1) node {$R_i^t$};
\end{tikzpicture}
\end{center}
\caption{Compatibility between the two possible edges of a given type $R_i$.}
\label{compatibilityRGE}
\end{figure}
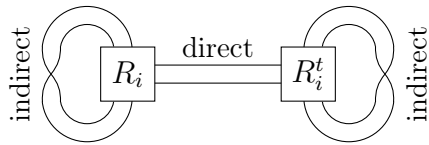

\item \textbf{Gaussian Unitary Ensemble (GUE).}  Edges corresponding to the GUE are usually not endowed with an orientation, as there is only one letter of type $H_i$. The only admissibility rule for a pairing is that it has to match a GUE letters with a GUE letter of the same type. The resulting identification of vertices is always direct (so that, for instance, words involving only the letters $H_i, G_i, G_i^*$ will yield only orientable surfaces). This is the ensemble on which genus expansion has been applied most often. It played a central role, for instance, in the first proof of Harer-Zagier formula. If one considers a more general setting involving letters $H_i$ and $\overline{H}_i$, the compatibility rule is summarized by Figure \ref{compatibilityGUE} below.

\begin{figure}[ht]
\begin{center}
\begin{tikzpicture}[scale=1.2]
\draw  (-.9,1.1) -- (-.3,1.1) to[bend left=20] (0,1) to[bend right=20] (.3,.9) -- (.9,.9);
\draw  (-.9,.9) -- (-.3,.9) to[bend right=20] (0,1) to[bend left=20] (.3,1.1)-- (.9,1.1);
 \draw (-1.45,1.25) circle (.3cm);
  \draw (-1.45,.75) circle (.3cm);
  \draw (-1.45,1.25) circle (.5cm);
  \draw (-1.45,.75) circle (.5cm);
  \filldraw[fill=white, draw=white] (-2,.7) rectangle (-1,1.3);
 \draw (-1.945,1.31) -- (-1.945,.69);
 \draw (-1.745,1.31) -- (-1.745,.69);
 \draw (1.45,1.25) circle (.3cm);
  \draw (1.45,.75) circle (.3cm);
  \draw (1.45,1.25) circle (.5cm);
  \draw (1.45,.75) circle (.5cm);
  \filldraw[fill=white, draw=white] (2,.7) rectangle (1,1.3);
 \draw (1.945,1.31) -- (1.945,.69);
 \draw (1.745,1.31) -- (1.745,.69);
\draw (0,1.3) node {\small indirect};
\draw (-2.2,1) node {
  \begin{tikzpicture}
      \node[rotate=90] {\small direct};    
    \end{tikzpicture}
};
\draw (2.2,1) node {
  \begin{tikzpicture}
 \node[rotate=90] {\small direct};    
 \end{tikzpicture}
};
\draw[fill=white] (-1-.3,1-.3) rectangle (-1+.3,1+.3);
\draw[fill=white] (1-.3,1-.3) rectangle (1+.3,1+.3);
\draw (-1,1) node {$H_i$};
\draw (1,1) node {$\overline{H_i}$};
\end{tikzpicture}
\end{center}
\caption{Compatibility between the two possible edges of a given type $H_i$.}
\label{compatibilityGUE}
\end{figure}
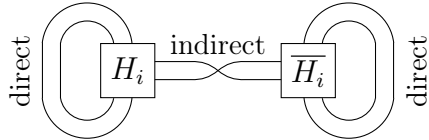

\item \textbf{Gaussian Orthogonal Ensemble (GOE).}

As for the GUE, edges corresponding to GOE letters are not endowed with an orientation, as there is only only letter of type $S_i$. The only admissibility rule for a pairing is that it has to match every GOE letter to a GOE letter of the same type; but the pairing comes with the specification of each matching to be meant in the direct or the indirect sense. (This is the only case in which an information has to be added to the pairing itself, so that the resulting surface is well defined). This is illustrated in Figure \ref{compatibilityGOE} below.

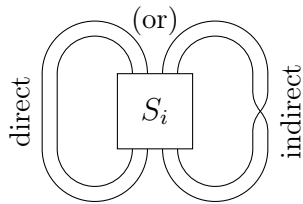
\begin{figure}[ht]
\begin{center}
\begin{tikzpicture}[scale=1]
 \draw (-.8,1.5) circle (.5cm);
  \draw (-.8,.5) circle (.5cm);
  \draw (-.8,1.5) circle (.7cm);
  \draw (-.8,.5) circle (.7cm);
  \filldraw[fill=white, draw=white] (-.4,.5) rectangle (-2,1.5);
 \draw (-1.3,1.51) -- (-1.3,.49);
 \draw (-1.5,1.51) -- (-1.5,.49);
 \draw (.8,1.5) circle (.5cm);
  \draw (.8,.5) circle (.5cm);
  \draw (.8,1.5) circle (.7cm);
  \draw (.8,.5) circle (.7cm);
  \filldraw[fill=white, draw=white] (.4,.5) rectangle (2,1.5);
 \draw (1.3,1.51) -- (1.3,1.3) to[bend right=20] (1.4,1) to[bend left=20] (1.5,.7) -- (1.5,.49);
 \draw (1.5,1.51) -- (1.5,1.3) to[bend left=20] (1.4,1) to[bend right=20] (1.3,.7) -- (1.3,.49);
\draw (-1.8,1) node {
  \begin{tikzpicture}
      \node[rotate=90] {\small direct};    
    \end{tikzpicture}
};
\draw (1.8,1) node {
  \begin{tikzpicture}
 \node[rotate=90] {\small indirect};    
 \end{tikzpicture}
};
\draw[fill=white] (-.5,1-.5) rectangle (.5,1.5);
\draw (0,1) node {$S_i$};
\draw (0,2.2) node {\small (or)};
\end{tikzpicture}
\end{center}
\caption{Compatibility between two edges of type $S_i$. Orientation, for each pair, has to be specified together with the pairing.}
\label{compatibilityGOE}
\end{figure}

\end{enumerate}

\subsubsection{Examples.}\label{examplesubsec3}
Many-letter words, as defined above, are not as natural as just $*$-words from the point of view of random matrices. However, they might prove to be useful if one is interested in combinatorial problems over non-orientable surfaces of low genus, such as the projective plane or the Klein bottle. We give below basic examples yielding these two celebrated non-orientable surfaces.

\begin{itemize}
\item If $w= G_1 G_2 \overline{G_1} \overline{G_2}$, the only admissible pairing yields a projective plane, as represented on Figure \ref{Boy_ex}, which is of non-orientable genus $1$. As a result,
    $$
    \E \Tr \lll( G_1 G_2 \overline{G_1} \overline{G_2} \rr) = 1.
    $$
    In addition, 
    \[
    \Tr G_w - 1 \distconv \mathscr{N}_{\R}(0,1),
    \]
    since $a_w=0$ and $b_w=c_w=p_w=1$ by a direct computation.
    
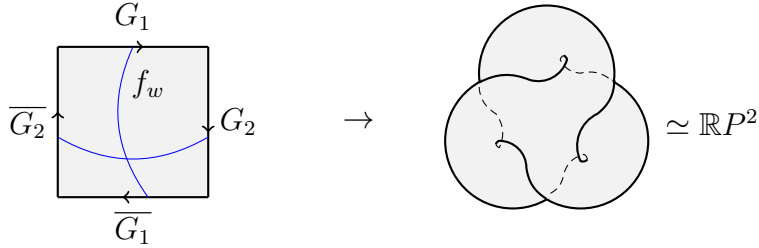
\begin{figure}[ht]
\begin{center}
\begin{tikzpicture}
\filldraw[fill=black!5!white,draw=white] (-1,-1) rectangle (1,1);
\begin{scope}[thick,decoration={
    markings,
    mark=at position 0.57 with {\arrow{>}}}
    ] 
\draw[postaction={decorate}] (-1,1) -- (1,1);
\draw[postaction={decorate}] (1,1) -- (1,-1);
\draw[postaction={decorate}] (1,-1) -- (-1,-1);
\draw[postaction={decorate}] (-1,-1) -- (-1,1);
\end{scope}
\draw (0,1.4) node {$G_1$};
\draw (0,-1.4) node {$\overline{G_1}$};
\draw (1.4,0) node {$G_2$};
\draw (-1.4,0) node {$\overline{G_2}$};
\draw (0.2,0.5) node {$f_{w}$};
\draw[blue] (-1,-.2) to[bend right] (1,-.2);
\draw[blue] (0,1) to[bend right] (.2,-1);
\draw (3,0) node {$\rightarrow$};
\draw (5.5,0.2) node {
\begin{tikzpicture}[scale=.9]
\filldraw[fill=black!5!white,thick] (0,.5) circle (1cm);
\filldraw[fill=black!5!white,thick] (.5,-.5) circle (1cm);
\filldraw[fill=black!5!white,thick] (-.5,-.5) circle (1cm);
\filldraw[fill=black!5!white,draw=black!5!white] (0,-.215) circle (1.135cm);
\draw[thick] (0,-1.36) to[bend left=25] (-.3,-.9) to[bend right=45] (-.7,-.5) to[bend right] (-.75,-.55);
\draw[densely dashed] (-.75,-.55) to[bend right] (-.7,-.6) to[bend right] (-.75,-.1) to[bend left] (-1,.35) ;
\draw[thick] (-1,.36) to[bend left=25] (-.25,.4) to[bend right=45] (0.3,.7) to[bend right] (0.25,.76);
\draw[densely dashed] (0.25,.76) to[bend right] (.2,.7) to[bend right] (0.5,.5) to[bend left] (1,.35) ;
\draw[thick] (1,.4) to[bend left=25] (.75,-.1) to[bend right=45] (.5,-.8) to[bend right] (.6,-.78);
\draw[densely dashed] (.6,-.78) to[bend right] (.55,-.7) to[bend right] (0.25,-1) to[bend left] (0,-1.3) ;
\end{tikzpicture}    
};
\draw (7.7,0) node {$\simeq \Proj$};
\end{tikzpicture}
\end{center}
\caption{Admissible pairing and resulting projective plane.}
\label{Boy_ex}
\end{figure}
    
\item If $w= G_1 G_2 G_1^* \overline{G_2}$, the only admissible pairing yields a Klein bottle, as represented on Figure \ref{Klein_bottle_ex}, which has non-orientable genus $2$. Therefore,
    $$
    \E \Tr \lll( G_1 G_2 G_1^* \overline{G_2} \rr) = \frac{1}{N}.
    $$
    
\begin{figure}[ht]
\begin{center}
\begin{tikzpicture}
\filldraw[fill=black!5!white,draw=white] (-1,-1) rectangle (1,1);
\begin{scope}[thick,decoration={
    markings,
    mark=at position 0.57 with {\arrow{>}}}
    ] 
\draw[postaction={decorate}] (-1,1) -- (1,1);
\draw[postaction={decorate}] (1,1) -- (1,-1);
\draw[postaction={decorate}] (-1,-1) -- (1,-1);
\draw[postaction={decorate}] (-1,-1) -- (-1,1);
\end{scope}
\draw (0,1.4) node {$G_1$};
\draw (0,-1.4) node {$G_1^*$};
\draw (1.4,0) node {$G_2$};
\draw (-1.4,0) node {$\overline{G_2}$};
\draw (0.2,0.5) node {$f_{w}$};
\draw[blue] (-1,-.2) to[bend right] (1,-.2);
\draw[blue] (0,1) to[bend right] (.2,-1);
\draw (3,0) node {$\rightarrow$};
\draw (5,0) node {
\begin{tikzpicture}[scale=.8]
\filldraw[fill=black!5!white,thick] (-.5,-1) to[bend right=60] (.5,-1) to[bend right] (.5,0) to[bend left] (.5,1) to[bend right] (.5,2) to[bend right] (0,2.3) to[bend right] (-.5,2) to[bend right] (-.5,1) to[bend left] (-.5,0) to[bend right] (-.5,-1);
\filldraw[fill=white,thick] (0,1) to[bend right] (.2,2) to[bend right=60] (-.2,2) to[bend right] (0,1);
\draw[thick] (0,1) to[bend left] (-.35,.5) ;
\draw[densely dashed] (0,1) to[bend right=30] (-.35,.5) ;
\draw[densely dashed] (-.35,.5) to[bend left] (-.3,-1) to[bend left] (-.5,-1) ;
\draw[densely dashed] (0,1) to[bend left=20] (.2,-.7) to[bend right] (.35,-1) to[bend right] (.5,-1) ;
\end{tikzpicture}    
};
\draw (6.5,0) node {$\simeq N_2$};
\end{tikzpicture}
\end{center}
\caption{Admissible pairing and resulting Klein bottle.}
\label{Klein_bottle_ex}
\end{figure}
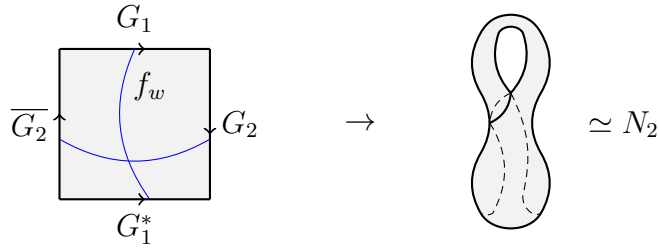

After checking every admissible pairing and their resulting surfaces, the parameters are found to be
$$a_w=0, \quad p_w=0, \quad b_w=1, \quad c_w=0,$$
so that the central limit theorem gives
$$
\Tr \lll( G_1 G_2 G_1^* \overline{G_2} \rr) \distconv \mathscr{N}_{\C}(0,1).
$$

\end{itemize}


\section*{Acknowledgments}
The authors would like to thank Gernot Akemann, Benson Au, Paul Bourgade, Jesper Ipsen, Camille Male, Jamie Mingo, Doron Puder, Emily Redelmeier, Roland Speicher, Wojciech Tarnowski and Ofer Zeitouni for useful discussions, comments and references, as well as the anonymous referee for a suggestion that greatly improved one of the theorems. \medskip

G.D. gratefully acknowledges support from the grants NSF DMS-1812114 of P. Bourgade (PI) and NSF CAREER DMS-1653602 of L.-P. Arguin (PI), as well as the European Union's Horizon 2020 research and innovation programme under the Marie-Sk{\l}odowska-Curie Grant Agreement No. 754411.

\end{document}